\documentclass[12pt]{amsart}
\usepackage{amssymb}
\usepackage[all]{xy}

\theoremstyle{plain}
\newtheorem{thm}{Theorem}[section]

\newtheorem*{thm1.1}{Theorem 1.1}
\newtheorem*{thmM}{Main Theorem}

\newtheorem{lem}[thm]{Lemma}
\newtheorem{cor}[thm]{Corollary}
\newtheorem{pro}[thm]{Proposition}

\theoremstyle{definition}

\newtheorem{rem}[thm]{Remark}

\newtheorem{defi}[thm]{Definition}
\newtheorem{exe}[thm]{Example}

\numberwithin{equation}{section}
\newcounter{elno}                

\newcommand{\la}{\lambda}

\newcommand{\Supp}{{\rm Supp}}

\newcommand{\Spec}{{\rm Spec \,}}

\newcommand{\Frac}{{\rm Frac \,}}

\newcommand{\ord}{{\rm ord}}

\renewcommand{\d}{\stackrel{\mbox{\scriptsize{$\bullet$}}}{}}

\newcommand{\boxtensor}{{\Box\kern-9.03pt\raise1.42pt\hbox{$\times$}}}

\newcommand{\propsubset}
{\mbox{$\textstyle{
\subseteq_{\kern-5pt\raise-1pt\hbox{\mbox{\tiny{$/$}}}}}$}}

\newcommand{\A}{{\mathbb A}}

\renewcommand{\P}{{\mathbb P}}
\newcommand{\Q}{{\mathbb Q}}
\newcommand{\R}{{\mathbb R}}

\newcommand{\V}{{\mathbb V}}

\DeclareMathOperator{\supp}{Supp}
\DeclareMathOperator{\SH}{SH}
\DeclareMathOperator{\SHP}{SH^+(V_\infty)}
\DeclareMathOperator{\Tan}{Tan}

\title[]{Intersection of valuation rings in $k[x,y]$}

\author{Xie Junyi}
\address{Centre de Math\'ematiques Laurent Schwartz \'Ecole Polytechnique, 91128, Palaiseau
Cedex, France}

\email{junyi.xie@math.polytechnique.fr}
\date{\today}
\thanks{The author is supported
by the ERC-starting grant project "Nonarcomp" no.307856.}

\begin{document}
\bibliographystyle{plain}
\begin{abstract}
We associate to any given finite set of valuations on the polynomial ring in two variables over an algebraically
closed field a numerical invariant whose positivity characterizes the case when  the intersection of their valuation rings has maximal
transcendence degree over the base fields.

As an application, we give a criterion for when an analytic branch at infinity in the affine plane
that is defined over a number field in a suitable sense is the branch of an algebraic curve.

%
\end{abstract}

\maketitle
\tableofcontents

\section{Introduction}

Let $R:=k[x,y]$ denote the ring of polynomials in two variables over an algebraically closed field $k$.
Given any finite set of valuations $S$ on $R$ that are trivial on $k$, we define
$R_S = \cap_{v \in S} \{ P \in R, \, v(P) \ge 0\}$ as the intersection of the
valuation rings of the elements in $S$ with $R$.
We obtain in this way a $k$-subalgebra of $R$, and it is a
natural question to ask for the transcendence degree of the fraction field of $R_S$ over $k$
which is an integer $\delta(S) \in \{0,1,2\}$.

Our main result is the construction of a symmetric matrix  $M(S)$
whose signature characterizes the case when
$\delta(S) = 2$. We should mention that when all valuations in $S$ are divisorial, this matrix $M(S)$ 
is the same as the matrix $\widetilde{\mathcal{M}}$ in \cite[Corollary 4.9]{Mondalb}.

As we shall see below, this construction is based
on the analysis developped by C. Favre and M. Jonsson \cite{Favre2011} on the tree of normalized
rank $1$ valuations
centered at infinity on $R$. In the case $S$ consists only of divisorial valuations, $M(S)$ can however be
defined using classical intersection theory on an appropriate projective compactification of
the affine plane, and we shall explain that one can recover in this way recent results
by Schroer \cite{Schroer2000} and Mondal \cite{Mondal}.

\bigskip

To get some insight into the problem, let us now describe a couple of examples.
We first observe that if $S_1,S_2$ are two finite sets of valuations satisfying $S_1\subseteq S_2$, then we have $R_{S_2}\subseteq R_{S_1}$. Also it is  only necessary to consider  valuations $v$ that are centered at infinity in the sense that  $R$ is not contained in the valuation ring of $v$.

We first recall the definition of a monomial valuation. Given $(s,t) \in \R^2 \setminus \{ (0,0)\}$, we denote by
$v_{s,t}:R\rightarrow \mathbb{R}$ the rank $1$ valuation defined by
\begin{equation}\label{eq:defmono}
v_{s,t} \left(\sum_{i,j\geq 0} a_{i,j}x^iy^j\right) :=  \min\left\{si+tj|\,\, a_{i,j}\neq 0\right\}~.
\end{equation}
The valuation $v_{s,t}$ is centered at infinity iff $\min \{ s, t \} <0$, and one immediately checks that
$R_{\{v_{s,t}\}}=k$ when $ \max \{ s, t \} < 0$ so that $\delta(\{v_{s,t}\}) = 0$ in this case.
This happens in particular when $(s,t) = (-1, -1)$ that is $\delta (\{- \deg\}) = 0$.

Fix a compactification $\mathbb{A}^2_k \subset \mathbb{P}^2_k$, and write $L_{\infty} = \mathbb{P}^2_k\setminus \mathbb{A}^2_k$ for the line at infinity.
Recall that a polynomial $P\in R$ is said to have one place at infinity, if the closure of $P =0$ intersects $L_{\infty}$
at a single point and the germ of curve it defines at that point is analytically irreducible. If $P$ has one place at infinity,   it follows
from a theorem of Moh~\cite{T.T.Mok1974} that all curves $\{ P = \lambda\}$ have  one place at infinity.
This pencil thus defines a rank $1$ (divisorial) valuation $v_{|P|}$ sending
$Q \in R$ to $v_{|P|} (Q) := \# \{ P^{-1} (\lambda) \cap Q^{-1}(0) \}$ for $\lambda$ generic.
One has in this case $R_{\{v_{|P|}\}}=k[P]$, hence $\delta (\{ v_{|P|} \}) = 1$.

To get examples of a finite family valuations such that $\delta = 2$, it is  necessary to
choose valuations that are far enough from $-\deg$. A first construction arises as follows.
Pick $s, t \in \R^2$ such that $s< 0 < t$ and let $m$ be any integer larger than $s$.
Since $ k [x y^m , y ] \subset R_{v_{s,t}}$ it follows that $\delta (\{v_{s,t}\}) =2$.

Next choose $\{s_i\}_{1 \le i \le m}$ any finite set of branches based at points lying on $L$ of algebraic curves defined in $\mathbb{A}^2$ by  equations $\{P_i=0\}$.  Let $v_i$ be the rank $2$ valuation on $R$ associated to the branch $s_i$.
Then one checks that $ (P_1 \cdots P_m). R \subset R_{\{v_1,\cdots, v_m\}}$ so that
$\delta( \{v_1,\cdots, v_m\}) =2$.

%

A first (simple) characterization of the case $\delta(S) =2$ is as follows.
\begin{thm}\label{thmrichexistpposi}
Let $S$ be any finite set of rank one valuations on $R = k[x,y]$ that are trivial on $k^*$.
Then the transcendence degree $\delta(S)$ of the fraction field of the intersection of $R$ with the valuation
rings of the valuations in $S$ is equal to $2$ iff there exists a polynomial
$P\in R$ satisfying $v(P)>0$ for all $v\in S$.
\end{thm}


\bigskip

We now describe  more precisely  our main result.
Since the construction of our matrix $M(S)$ relies on the fine tree structure of
the space of normalized
rank $1$ valuations centered at infinity (see Section \ref{sectionvatree}), we first explain our main theorem in the
simplified (yet important)
situation when all valuations are divisorial.

\medskip

Now pick any proper modification $\pi : X\to \mathbb{P}^2$ that is an isomorphism above the affine plane with $X$ a smooth projective surface.
Let $\{E_0, E_1,\cdots,E_m\}$ be the set of all irreducible components of $X\setminus \mathbb{A}^2_k$ with $E_0$ the strict transform of $L_{\infty}$, and $S$ be a subset of $\{\ord_{E_0}, \ord_{E_1},\cdots,\ord_{E_m}\}$.

Since the intersection form on the divisors $E_i$'s is non-degenerate, for each $i$, there exists a unique divisor $\check{E}_i$
supported at infinity such that $(\check{E}_i\cdot E_j)=\delta_{i,j}$ for all $i,j$. Observe that
$(\check{E}_0\cdot\check{E}_0) = +1 >0$.

%
%
%

Finally we define $M(S)$ to be the symmetric matrix whose entries are given by $[(\check{E}_i\cdot \check{E_j})]_{1 \le i,j \le m}$.

%
%
%
Our main theorem in the case of divisorial valuations reads as follows.
\begin{thm}\label{thmmainthmchifordiv}
Given any finite set of divisorial valuations $S$ on $R$ that are centered at infinity,
we have $\delta(S) =2$  if and only if the matrix $M(S)$ is negative definite.
\end{thm}
By Hodge index  theorem, the matrix $M(S)$ is negative definite if and only if $\chi(S):=(-1)^m\det M(S)>0$.

When $S$ is reduced to a singleton, Theorem \ref{thmmainthmchifordiv}  is due to P. Mondal, see \cite[Theorem 1.4]{Mondal}.


\bigskip

To treat the case of not necessarily divisorial valuations we need to briefly recall some facts on
the valuation tree as defined by C. Favre and M. Jonsson (see Section~\ref{sectionvatree} for details).

We denote by $V_{\infty}$
the set of functions
$$v:k[x,y]\rightarrow \mathbb{R}\cup \{+ \infty\}$$
that satisfy the axiom of valuations $v(PQ) = v(P) + v(Q)$, and $v(P+Q) \ge \min\{ v(P) , v(Q)\}$ and normalized by
$\min\{v(x),v(y)\}=-1$. However, we allow $v$ to take the value $+\infty$ on a non-constant polynomial.
The set $V_\infty$ is a compact topological space when equipped with the topology of the pointwise convergence.
It can be also endowed with a natural partial order relation given by $v \le v'$ if and only if $v(P) \le v'(P)$ for all $P \in R$.
The unique minimal point for that order relation is $-\deg$, and $V_\infty$ carries a tree structure in the sense that
for any $v'$ the set $\{v\in V_{\infty}|\,  -\deg \le v \le v'\}$ is isomorphic as a poset to a segment in $\R$ with its standard order relation.
In particular, one may define the minimum $ v \wedge v'$ of  any two valuations $v, v'\in V_\infty$.

\smallskip

There is a canonical way to associate an element $\bar{v}\in V_\infty$
to a given valuation $v$ on $R$ that is trivial on $k$.
When $v$ has rank $1$, we may assume it takes its values in $\R$, and
$\bar{v}$ is the unique valuation that is proportional to $v$ and normalized by $\min\{ \bar{v} (x), \bar{v}(y) \}= -1$.
For instance when $E$ is an irreducible component of
$\pi^{-1}(L)$ for some proper modification $\pi  : X \to \mathbb{P}^2$ as above,
then we define $b_E:=\min\{\ord_E(x),\ord_E(y)\}$, and we have $v_E = \frac1{b_E} \ord_E \in V_\infty$.
When $v$ has rank $2$ and is associated to a branch $s$ at infinity of  an irreducible curve
at infinity
$C$ in $\mathbb{A}^2$, then $\bar{v}(P)$ is  the local intersection number  of $s$ with the divisor
of $P$ with the convention that $\bar{v}(P) = +\infty$ when $P$ vanishes on $C$.
Finally when $v$ has rank $2$ and its valuation ring contains the
valuation ring of a divisorial valuation centered at infinity, we set $\bar{v}$  to be this divisorial valuation.

\smallskip

%
%

The skewness function $\alpha:V_{\infty}\rightarrow [-\infty,1]$ is the unique upper semicontinuous function on $V_{\infty}$
that is decreasing along any segment starting from $-\deg$, and that
satisfies $\alpha(v_E)=b_E^{-2}(\check{E}\cdot \check{E})$ for any divisorial valuation (in the notation introduced above).
On the other hand, $\alpha (v) = - \infty$ when $v$ is associated to a branch at infinity of an algebraic curve in $\mathbb{A}^2$.

%

\smallskip

Now given any finite subset $S=\{v_1,\cdots,v_m\}$ of valuations centered at infinity and trivial on $k$,
we let $\bar{S}= \{ \bar{v}, \, v \in S\}\subset V_\infty$
and define
\begin{equation}\label{eq:defchi}
M(\bar{S}):=[\alpha(\bar{v_i}\wedge \bar{v_j})]_{1\leq i,j\leq m}.
\end{equation}
This is a symmetric matrix with entries in $\mathbb{R}\cup \{-\infty\}.$

As above, we then have
\begin{thmM}
Given any finite set of valuations $S$ on $R$ that are trivial on $k$ and centered at infinity,
we have $\delta(S) =2$  if and only if $M(\bar{S})$ is negative definite.
\end{thmM}
When one entry of the matrix $\alpha(\bar{v_i}\wedge \bar{v_j})$ is equal to $-\infty$, we say that $M(\bar{S})$ is negative definite if and only if the matrix
$[(\max\{\alpha(\bar{v_i}\wedge \bar{v_j}),-t\}]_{1\leq i,j\leq m}$ is negative definite for $t$ large enough.

Observe that one can use  Hodge index theorem to characterize the case when
$M(\bar{S})$ is negative definite by a {\em numerical invariant} $\chi(\bar{S}):=(-1)^l \det M(\bar{S})$. Here $l$ denotes the cardinality of
$\bar{S}$ and
$\det(M(\bar{S})) := \lim_{t\rightarrow-\infty}\det(\max\{\alpha(\bar{v_i}\wedge \bar{v_j}),t\})_{1\leq i,j\leq m}$
when one entry of the matrix $\alpha(\bar{v_i}\wedge \bar{v_j})$ is equal to $-\infty$.
Observe that the limit exists because the quantity $\det(\max\{\alpha(\bar{v_i}\wedge \bar{v_j}),t\})_{1\leq i,j\leq m}$ is a polynomial for $t$ large enough.

Indeed
our Main Theorem can be phrased by saying that $\delta(S) =2$
if and only if $\chi(\bar{S})>0$.


When  $S$ contains only one point $v$, we get $M(S)=\alpha(v)$ and Theorem \ref{thmrichexistpposi} together with our Main Theorem imply the following result of P. Mondal.
\begin{thm}[\cite{Mondal}]
For a  valuation $v\in V_{\infty}$,
the existence of  a non constant polynomial $P\in k[x,y]$ such that $v(P)>0$ is equivalent to
$\alpha(v)<0$.
\end{thm}

\smallskip

Our Main Theorem also implies the following
\begin{cor}
Let $s_1,\cdots,s_m$ be a finite set of formal branches of curves centered at infinity.
Then there exists a polynomial $P\in k[x,y]$ such that $\ord_{\infty}(P|_{s_i})>0$ for all $i=1,\cdots,m$.
\end{cor}

\medskip

In a sequel to this paper \cite{Xiec}, we shall use these results to get a proof of the dynamical Mordell-Lang conjecture
for polynomial endomorphisms on $\mathbb{A}^2_{\overline{\mathbb{Q}}}$.

\medskip

We conclude this introduction by giving  a criterion of arithmetic nature for an analytic branch at infinity to be algebraic.

The setting is as follows.
Let $K$ be a number field. For any finite set $S$ of places of $K$ containing all archimedean places, denote by $O_{K,S}$ the ring of $S$-integers in $K$. For any place $v$ on $K$, denote by $K_v$ the completion of $K$ w.r.t. $v$.
We cover the line at infinity $L_{\infty}$ of the compactification of $\mathbb{A}^2_K = \Spec K[x,y]$ by $\mathbb{P}^2_K$ by charts
$U_q=\Spec K[x_q,y_q]$ centered at $q\in L_{\infty}(K)$ so that $q  = \{ (x_q,y_q) = (0,0)\}$,  $L_{\infty} \cap U_i = \{ x_i =0\}$,
and $x_q = 1/x$, $y_q = y/x + c$ for some $c\in K$ (or $x_q = 1/y$, $y_q = x/y$).

We shall say that $s$ is an \emph{adelic} branch defined over $K$ at infinity if it is given by the following data.
%
\begin{points}
\item
$s$ is a formal branch based at a point $q\in L_{\infty}(K)$  given in coordinates $x_q, y_q$ as above by a formal Puiseux series
$y_q=\sum_{j\ge 1}a_{j}x_q^{j/m}\in O_{K,S}[[x_q^{1/m}]]$  for some positive integer $m$ and some finite set $S$ of places of $K$ containing all archimedean places.
\item
for each place $v\in S$, the radius of convergence of the Puiseux series determining $s$ is positive, i.e.
$\limsup_{j\rightarrow\infty}|a_j|_v^{-m/j}>0$.
\end{points}
Observe that for any other place $v\notin S$, then the radius of convergence is a least $1$. In the sequel, we set
$r_{C,v}$ to be the minimum between $1$ and the radius of convergence over $K_v$ of this Puiseux series.

Any adelic branch $s$ at infinity thus defines an analytic curve $$C^v(s):=\{(x_i,y_i)\in U_i(K_v)|\,\,y_i=\sum_{j=1}^{\infty}a_{ij}x^{\frac{j}{m_i}}, |x_i|_v< \min\{r_{C_i,v},1\}\}.$$


\begin{thm}\label{thmanalytictoalgemanyrational}
Suppose $s_1,\cdots, s_l$, $l\geq 1$ is a finite set of adelic branches at infinity.
Let $\{B_v\}_{v\in M_K}$ be a set of positive real numbers such that
$B_v=1$ for all but finitely many places.

Finally let $p_n = (x^{(n)}, y^{(n)})$, $n\geq 0$ be an infinite collection of $K$-points in $\mathbb{A}^2(K)$ such that for each place $v \in M_K$ then either
$\max\{|x^{(n)}|_v,|y^{(n)}|_v\}\leq B_v$ or $p_n \in \cup_{i=1}^lC^v(s_i)$.

Then there exists  an algebraic curve $C$ in $\mathbb{A}^2_K$
such that  any branch of $C$ at infinity is contained in the set $\{s_1,\cdots, s_l\}$ and $p_n$ belongs to $C(K)$ for all $n$ large enough.

In particular, by the theorem of Faltings \cite{Faltings1994}, the geometric genus of $C$ is at most one.
\end{thm}

\bigskip

The article is organized in five sections. Section \ref{sectionvatree} contains background
informations on the valuation tree $V_\infty$.
Section \ref{sectionpotentialtheory} is entirely
devoted to the description of a potential theory  in $V_\infty$. Especially important for us are the notion of subharmonic functions and
the definition of a Dirichlet energy.
\label{sectionrzspaceandv}
The proof of our main theorem can
be found in Section \ref{sectionpolytakingnova}.
Section \ref{sectionrechisez} contains various remarks in the case
$\delta = 0$ or $1$.
 Finally Section \ref{sectionapptoalg} contains the proof of Theorem \ref{thmanalytictoalgemanyrational}.
\newpage

\section{The valuation tree}\label{sectionvatree}
Let $k$ be any algebraically closed field. In this section, we recall some basic facts on
the space of normalized valuations centered at infinity in the affine plane and its tree structure following \cite{Favre2004,Favre2007,Favre2011,Jonsson}.

\subsection{Definition}
The set  $V_{\infty}$ is defined as the set of functions
$v:k[x,y]\rightarrow (-\infty ,+\infty]$ satisfying:
\begin{points}
\item $v(P_1P_2)=v(P_1)+v(P_2)$ for all $P_1,P_2\in k[x,y]$;
\item $v(P_1+P_2)\geq \min\{v(P_1),v(P_2)\}$;
\item $v(0)=+\infty$, $v|_{k^*}=0$ and $\min\{v(x),v(y)\}=-1$.
\end{points}
We endow  $V_{\infty}$ with the topology of the pointwise convergence, for which it is  a compact space.

Given $v\in V_\infty$, the set $ \mathfrak{P}_v:= \{ P, v(P) = +\infty\}$ is a prime ideal.
When it is reduced to $(0)$ then $v$ is a rank $1$ valuation on $k[x,y]$.
Otherwise it is generated by an irreducible polynomial $Q$, and for any $P\in k[x,y]$
the quantity $v (P)$ is the order of vanishing of $P|_Q$ at a branch of the curve $Q^{-1}(0)$ at infinity with the convention $v(P) = +\infty$ when $P\in
\mathfrak{P}_v$.

Let $s$ be a formal branch of curve centered at infinity. We may associate to $s$ a valuation $v_s\in V_{\infty}$ defined by $P\mapsto -\min\{\ord_{\infty}(x|_s),\ord_{\infty}(y|_s)\}^{-1}\ord_{\infty}(P|_s)$. Such a valuation is called a curve valuation.

\smallskip

Suppose $X$ is a smooth projective compactification of $\A^2_k$.
The center of $v\in V_\infty$
in $X$ is the unique scheme-theoretic point on $X$ such that its associated valuation is
strictly positive on the maximal ideal of its local ring.
A divisorial valuation is an element $v\in V_\infty$ whose center has codimension $1$
for at least one compactification $X$ as above.

More precisely, let  $E$ be an irreducible divisor of $X\setminus \A^2_k$.
Then  the order of vanishing $\ord_E$ along $E$  determines a divisorial valuation on $k[x,y]$, and $v_E:=(b_E)^{-1}\ord_E \in V_\infty$
where $b_E:=-\min\{\ord_E(x),\ord_E(y)\}$.

\smallskip

\noindent {\bf Warning}.
In the sequel, we shall refer to elements in $V_\infty$ as \emph{valuations}
even when the prime ideal  $ \mathfrak{P}_v$ is non trivial.

%
%
%

\subsection{The canonical ordering and the tree structure}

The space $V_{\infty}$ of normalized valuations is
equipped with a partial ordering defined by $v\leq w$ if and only if $v(P)\leq w(P)$ for all $P\in k[x,y]$ for which  $-\deg$ is the unique minimal element.

All curve valuations are maximal and and no divisorial valuation is maximal.

It is a theorem that given any valuation $v\in V_\infty$ the set
$ \{ w \in \V_\infty, \, - \deg \le w \le v \}$ is isomorphic as a poset to the real segment
$[0,1]$ endowed with the standard ordering. In other words, $(V_\infty, \le)$ is a rooted tree in the sense of \cite{Favre2004,Jonsson}.

It follows that given any two valuations $v_1,v_2\in V_{\infty}$,
there is a unique valuation in $V_{\infty}$ which is maximal in the set $\{v\in V_{\infty}|\,\,v\leq v_1 \text{ and } v\leq v_2\}.$ We denote it by  $v_1\wedge v_2$.

The segment $[v_1, v_2]$ is by definition the union of $\{w , \, v_1\wedge v_2 \le w \le v_1\}$
and $\{w , \, v_1\wedge v_2 \le w \le v_2\}$.

\smallskip

Pick any valuation  $v\in V_\infty$. We say that two points $v_1, v_2$
lie in the same direction at $v$ if the segment $[v_1, v_2]$ does not contain $v$.
A direction (or a tangent vector) at $v$ is an equivalence class for this relation.
We write $\Tan_v$ for the set of directions at $v$.

When $\Tan_v$ is a singleton, then $v$ is called an endpoint. In $V_\infty$, the set of endpoints is exactly the set of all maximal valuations. This set is dense in $V_{\infty}.$

When $\Tan_v$ contains exactly two directions, then $v$ is said to be regular.
In $V_\infty$, regular points are given by monomial rank $1$ valuations as in \eqref{eq:defmono}  for which the weights are rationally independent, see \cite{Favre2004,Jonsson} for details.

When $\Tan_v$ has more than three directions, then $v$ is a branched point. In $V_\infty$, branched points are exactly the divisorial valuations. Given any smooth projective compactification $X$ in which $v$ has codimension $1$ center $E$, one proves that the
map sending an element $V_\infty$ to its center in $X$ induces a  map $\Tan_v \to E$ that is a bijection.

\smallskip

Pick any $v\in V_\infty$.
For any tangent vector $\vec{v}\in \Tan_v$, we denote by $U(\vec{v})$ the subset of those elements in $V_\infty$ that determine $\vec{v}$. This is an open set whose boundary is reduced to the singleton $\{v\}$. The complement of $\{w\in V_\infty, \, w \ge v\}$ is equal to $U(\vec{v}_0)$ where $\vec{v}_0$ is the tangent vector determined by $-\deg$.

It is a fact that finite intersections of open sets of the form $U(\vec{v})$ form a basis for the topology of $V_\infty$.

\smallskip

Finally recall that the \emph{convex hull} of  any subset $S\subset V_\infty$ is defined the set of valuations $v\in V_\infty$ such that
there exists a pair $v_1 , v_2 \in S$ with $v\in [v_1, v_2]$.

A \emph{finite subtree} of $V_\infty$ is by definition  the convex hull of a finite collection of points in $V_\infty$. A point in a finite subtree  $T\subseteq V_{\infty}$ is said to be an end point if it is maximal in $T.$

%
%
%
%

%
%

\subsection{The valuation space as the universal dual graph}
One can understand the tree structure of $V_\infty$ from the geometry of compactifications of $\A^2_k$ as follows.

\smallskip

Pick any  smooth projective compactification $X$ of $\mathbb{A}^2_k$.
The divisor at infinity $X\setminus \mathbb{A}^2_k$ has simple normal crossings, and we denote by  $\Gamma_X$ its dual graph: vertices are in bijection with irreducible components of the divisor at infinity, and vertices are joined by an edge when their corresponding component intersect at a point.

The choice of coordinates $x,y$ on $\mathbb{A}^2_k$ determines a privileged compactification
$\mathbb{P}^2_k$ for which the divisor at infinity is a rational curve $L_\infty$ and
$\ord_{L_{\infty}} = -\deg$. In this case, the dual graph is reduced to a singleton.

\smallskip

For a general compactification $X$, we may look at the convex hull (in $V_\infty$) of the finite set of valuations $v_E$ where $E$ ranges over all irreducible components of
$X\setminus \mathbb{A}^2_k$. It is a fact that the finite subtree that we obtain in this way is a geometric realization of the dual graph $\Gamma_X$. To simplify notation, we shall identify
$\Gamma_X$ with its realization in $V_\infty$. Observe that the dual graph $\Gamma_X$ inherits a partial order relation from its inclusion in $V_\infty$.

\smallskip

There is also a canonical retraction map $r_X: V_\infty \to \Gamma_X$
sending a valuation $v\in V_\infty$ to the unique $r_X(v)\in \Gamma_X$ such that $[r_X(v), v] \cap \Gamma_X = \{ r_X(v)\}$.

Say that a compactification $X'$ dominates another one $X$ when the canonical birational map
$X' \dashrightarrow X$ induced by the identity map on $\A^2_k$ is regular.
The category $\mathcal{C}$ of all smooth projective compactifications of $\A^2_k$ is an inductive set for this domination relation, and one can form the projective limit
$\Gamma_\mathcal{C}:= \varprojlim_{X \in \mathcal{C}} \Gamma_X$ using the retraction maps.
In other words, a point in $\Gamma_\mathcal{C}$ is a collection of points $v_X \in \Gamma_X$ such that
$r_X(v_{X'}) = v_X$ as soon as $X'$ dominates $X$.

It is a theorem that $\Gamma_\mathcal{C}$ endowed with the product topology is homeomorphic to $V_\infty$.

\smallskip

\noindent
{\bf Warning}.
In the sequel, we shall mostly consider smooth projective compactifications that \emph{dominates}
$\mathbb{P}^2_k$, and refer to them as \emph{admissible} compactifications of the affine plane.

\smallskip

Observe that $\Gamma_X$ contains $-\deg$ when $X$ is an admissible compactification.

\subsection{Parameterization}

The \emph{skewness} function
$\alpha:V_{\infty}\rightarrow [-\infty,1]$ is the function on $V_{\infty}$
that is strictly decreasing (for the order relation of $V_\infty$) satisfying $\alpha(-\deg)=1$ and
$$|\alpha(v_E)-\alpha(v_{E'})|=\frac1{b_Eb_{E'}}~.$$
whenever $E$ and $E'$ are two irreducible components of $X\setminus \A^2_k$ that intersect at
a point in some admissible compactification $X$ of the affine plane.

Since divisorial valuations are dense in any segment $[-\deg, v]$ it follows that $\alpha$ is uniquely determined by the conditions above. One knows that
$\alpha (v) \in \Q$ for any divisorial valuation, that $\alpha (v) \in \R\setminus \Q$
for any valuation that is a regular point of $V_\infty$, and that $\alpha (v) = -\infty$ for any curve valuation. However there are endpoints of $V_\infty$ with finite skewness.

There is a geometric interpretation of the skewness of a divisorial valuation as follows.
Let $X$ be an admissible compactification of $\mathbb{A}^2_k$, and $E$ be an irreducible component of $X\setminus \A^2_k$. Let $\check{E}$ be the unique divisor supported on the divisor at infinity such that  $(\check{E}\cdot E)=1$ and $(\check{E}\cdot F)=0$ for all components $F$ lying at infinity. Then we have $$\alpha(v_E)=\frac1{b_E^2}\,(\check{E}\cdot \check{E})~.$$

 Since the skewness function is strictly decreasing, it induces a metric $d_{V_\infty}$ on $V_\infty$ by setting
$$d_{V_\infty}(v_1,v_2):=2\alpha(v_1\wedge v_2)-\alpha(v_1)-\alpha(v_2)$$
for all $v_1, v_2\in V_\infty.$
In particular, any segment in $V_\infty$ carries a canonical metric for which it becomes isometric to a real segment.

\newpage

\section{Potential theory on $V_{\infty}$}\label{sectionpotentialtheory}
As in the previous section $k$ is any algebraically closed field.
We recall the basic principles of a potential theory on $V_\infty$ including the definition of subharmonic functions, and their associated Laplacian. We then construct a Dirichlet pairing on
subharmonic functions and study its main properties.

We refer to \cite{Jonsson} for details.

\subsection{Subharmonic functions on $V_\infty$}
To any $v\in V_\infty$ we attach its Green function
$$
g_v(w) := \alpha(v \wedge w)~.
$$
This is a decreasing continuous function taking values in $[-\infty, 1]$, satisfying
$g_v(-\deg) =1$. Moreover pick any $v'\in V_\infty$
and define the function $g(t) :  [\alpha(v'), 1] \to [-\infty, 1]$ by sending $t$
to $g(v_t)$ where $v_t$ is the unique valuation in $[-\deg , v']$ with skewness $t$.
Then $g$ is a piecewise affine increasing and convex function with slope in $\{0,1\}$.

\smallskip

Denote by  $M^+(V_\infty)$  the set of positive Radon measures on $V_\infty$
that is the set of positive linear functionals on the space of continuous functions on $V_\infty$.
We endow $M^+(V_\infty)$ with the weak topology.

\begin{lem}\label{lem714favvalu}For any positive Radon measures $\rho$ on $V_{\infty}$, there exists a sequence of compactification $X_n\in \mathcal{C}$, $n\geq 0$ such that $X_{n+1}$ dominates $X_{n}$ for all $n\geq 0$, and $\rho$ is supported on the closure of $\cup_{n\geq 0}\Gamma_{X_n}.$
\end{lem}
\proof Observe that $V_{\infty}$ is complete rooted nonmetric tree and weakly compact (See \cite[Section 3.2]{Favre2004}), thus \cite[Lemma 7.14]{Favre2004} apples. By \cite[Lemma 7.14]{Favre2004}, there exists a sequence of finite subtree $T_n$ $n\geq 0$ satisfying $T_{n}\subseteq T_{n+1}$ for $n\geq 0$ such that $\rho$ is supported on the closure $T$ of $\cup_{n\geq 0}T_n.$
Since $T_n$ is a finite tree and the divisorial valuations are dense in $T_n$, there exists a sequence of subtrees $T_n^m$ such that
\begin{points}
\item[$\d$] all vertices in $T_n^m$ are divisorial;
\item[$\d$] $T_{n}^m\subseteq T_{n}^{m+1}$ for $m\geq 0$;
\item[$\d$] $T_n$ is the closure of  $\cup_{m\geq 0}T_n^m.$
\end{points}
Set $Y_n:=\cup_{1\leq i,j\leq n} T_i^j$, then we have
\begin{points}
\item[$\d$] $Y_n$ is a finite tree;
\item[$\d$] all vertices in $Y_n$ are divisorial;
\item[$\d$] $Y_n \subseteq Y_{n+1}$ for $n\geq 0$;
\item[$\d$] $T$ is the closure of  $\cup_{n\geq 0}Y_n.$
\end{points}
To conclude, we pick by induction a sequence of increasing compactification $X_n\in\mathcal{C}$ such that $Y_n\subseteq \Gamma_{X_n.}$
\endproof

\begin{lem}\label{lemapromesbysupfini}Let $\rho$ be any positive Radon measures on $V_{\infty}$ and $T_n$ be a sequence of finite subtree of $V_{\infty}$ such that $T_n\subseteq T_{n+1}$ for $n\geq 0$ and $\rho$ is supported on the closure of $\cup_{n\geq 0}T_n.$ Then we have $r_{T_n*}\rho\to\rho$ weakly.
\end{lem}
\proof Let $T$ be the closure of $\cup_{n\geq 0}T_n$ and $f$ be any continuous function on $V_{\infty}$. For any $\varepsilon>0$ and any point $v\in T$, there exists a neighborhood $U_v$ of $v$ such that
$\sup_{U_v} | f - f(v)| \le \varepsilon/2$.
We may moreover choose it such that either $U_v = \{ w, w > w_1\}$ or  $U_v = \{ w,  w_1 < w\wedge w_2 < w_2\}$.
Since $T$ is compact,
it is covered by finitely many
such open sets $U_{v_1}, \ldots, U_{v_m}$. Since $\cup_{n\geq 1} T_n$ is dense in $T$, for any $i=1,\cdots,m$, there exists $w_i\in U_{v_i}\cap (\cup_{n\geq 1} T_n).$ There exists $N\geq 0,$ such that $T_N$ contains $\{w_1,\cdots,w_m\}$. For any $n\geq N$, if $v$ is a point in $U_{v_i}$, we have $r_{T_n}v\in U_{v_i}$. It follows that for all points $v\in T$, we have $|f(v)-f(r_{T_n})(v)|\leq \varepsilon$ and
$$\Big|\int_{V_{\infty}}f(v)d\rho(v)-\int_{V_{\infty}}f(v)dr_{T_n*}\rho(v)\Big|=\Big|\int_{T}f(v)-f(r_{T_n}(v))d\rho(v)\Big|\leq \varepsilon\rho(V_{\infty})$$ which concludes the proof.
\endproof

\smallskip

Given any positive Radon measure $\rho$ on $V_\infty$ we define
$$
g_\rho (w) := \int_{V_\infty} g_v(w) \, d\rho(v)~.
$$
Observe that $g_v(w)$ is always well-defined in $[-\infty, 1]$ since
$g_v \le 1$ for all $v$.
Since the Green function $g_v$ is decreasing for all $v\in V_{\infty}$, we get
\begin{pro}\label{proshdecreasing}For any any positive Radon measure $\rho$ on $V_\infty$, $g_{\rho}$ is decreasing.
\end{pro}

The next result is
\begin{thm}\label{thmrhotogrhoinj}
The map $\rho \mapsto g_\rho$ is injective.
\end{thm}
To prove this theorem, we first need the following

\begin{lem}\label{lemaproximationfbyfx}For any continuous function $f:V_{\infty}\rightarrow \mathbb{R}$ and any $\varepsilon>0$, there exists $X\in \mathcal{C}$ such that $|f-f\circ r_X|\leq \varepsilon$.
\end{lem}

\proof[Proof of Lemma \ref{lemaproximationfbyfx}]

For any $v$ we may find a neighborhood $U_v$ such that
$$\sup_{U_v} | f - f(v)| \le \varepsilon/2.$$ We may moreover choose it
such that $U_v = \{ w, w > w_1\}$ or  $U_v = \{ w,  w_1 < w
\wedge w_2 < w_2\}$
where $w_1, w_2$ are divisorial. Since $V_\infty$ is compact it is
covered by finitely many
such open sets $U_{v_1}, \ldots, U_{v_m}$. Choose $X$ to be an
admissible compactification
such that  the boundary valuations of $U_{v_i}$ has all codimension $1$ center in
$X$. For any $v\in V_\infty$
pick an index $i$ such that $v \in U_{v_i}$. Then we have
$|f(v)-f\circ r_X(v)|\leq |f(v)-f(v_i)|+|f(r_X(v))-f(v_i)|<\varepsilon.$
This concludes the proof.
\endproof

\begin{proof}[Proof of Theorem \ref{thmrhotogrhoinj}]By contradiction, suppose that $\rho_1\neq \rho_2$ in  $M^+(V_\infty)$ but $g_{\rho_1}=g_{\rho_2}$. There exists a continuous function $f:V_{\infty}\rightarrow\mathbb{R}$ satisfying $$\int_{V_{\infty}}f(v)d\rho_1(v)\neq \int_{V_{\infty}}f(v)d\rho_2(v).$$ Set $M:=\max\{\rho_1(V_{\infty}),\rho_2(V_{\infty})\}$.

By Lemma \ref{lemaproximationfbyfx}, for any $\varepsilon>0$, there exists $X\in \mathcal{C}$ such that $|f\circ r_X-f|\leq \varepsilon/2$.
 There exists a piecewise linear function $h$ on $\Gamma_X$ such that $|f\circ r_X-h\circ r_X|\leq \varepsilon/2$.
 Since $\Gamma_X$ is a finite graph, there exists $v_1,\cdots,v_m\in \Gamma_X$ such that $h\circ r_X=\sum_{i=1}^mr_ig_{v_i}$ where $r_1,\cdots,r_m\in \mathbb{R}$.

Since $g_{\rho_1}(v_i)=g_{\rho_2}(v_i)$ for $i=1,\cdots,m$, we have
$$\int_{V_{\infty}}h\circ r_X(v)d\rho_1(v)=\int_{V_{\infty}}\sum_{i=1}^mr_ig_{v_i}(v)d\rho_1(v)=\sum_{i=1}^mr_i\int_{V_{\infty}}g_{v}(v_i)d\rho_1(v)$$
$$=\sum_{i=1}^mr_ig_{\rho_1}(v_i)=\sum_{i=1}^mr_ig_{\rho_2}(v_i)=\int_{V_{\infty}}h\circ r_X(v)d\rho_2(v).$$
It follows that $$|\int_{V_{\infty}}f(v)d\rho_1(v)-\int_{V_{\infty}}f(v)d\rho_2(v)|\leq 2\varepsilon M.$$
We obtain a contradiction by letting $\epsilon \to 0$.
\end{proof}
One can thus make the following definition.
\begin{defi}
A function $\phi : V_\infty \to \R \cup \{-\infty\}$ is said to be \emph{subharmonic} if there exists a positive Radon measure $\rho$ such that $\phi = g_\rho$. In this case, we write
$\rho = \Delta \phi$ and call it the \emph{Laplacian} of $\phi$.
\end{defi}
Denote by $\SH$ (resp. $\SHP$) the space of subharmonic functions on $V_\infty$ (resp. of non-negative subharmonic functions on $V_\infty$).

\smallskip


\begin{pro}\label{prosubhapbyrestogx}For any subharmonic function $\phi$ on $V_\infty$, there exists a sequence of compactification $X_n\in \mathcal{C}$, $n\geq 0$ such that $X_{n+1}$ dominates $X_{n}$ for all $n\geq 0$, and $\phi=\lim_{n\to\infty}\phi\circ r_{X_n}$ pointwise.
\end{pro}

\proof
Write $\rho$ for $\Delta\phi.$ \ref{proshdecreasing}
Pick $X_n$ as in Lemma \ref{lem714favvalu}. By Lemma \ref{lemapromesbysupfini}, $r_{X*}\rho\to\rho$ weakly.
For any $w\in V_{\infty}$, pick a sequence $w_n\in [-\deg,w]$ satisfying $w_n\to w$ when $n\to \infty.$
$$g_{\rho}(w)= \int_{V_\infty} g_v(w) \, d\rho(v)=\lim_{m\to\infty} \int_{V_\infty} g_v(w_m) \, d\rho(v)$$$$=\lim_{m\to\infty} \lim_{n\to\infty} \int_{V_\infty} g_v(w_m) \, dr_{X_n*}\rho(v).   \eqno (1)$$
Observe that $\int_{V_\infty} g_v(w_m) \, dr_{X_n*}\rho(v)=\int_{V_\infty} g_v(r_{X_n*}(w_m)) \, d\rho(v)$ which is decreasing in $n$ and $m$. We have
$$g_{\rho}(w)=\lim_{n\to\infty} \lim_{m\to\infty} \int_{V_\infty} g_v(w_m) \, dr_{X_n*}\rho(v)=\lim_{n\to\infty}\int_{V_\infty} g_v(w) \, dr_{X_n*}\rho(v)$$
$$=\lim_{n\to\infty}\int_{V_\infty} g_v(r_{X_n}w) \, d\rho(v)
=\lim_{n\to\infty}g_{\rho}\circ r_X(w).$$
\endproof

\subsection{Reduction to finite trees}
Let $T$ be any finite subtree of $V_\infty$ containing $-\deg$. Denote by $r_T: V_\infty \to T$ the canonical retraction defined by sending $v$ to the unique valuation $r_T(v)\in T$ such that
$[r_T(v), v] \cap T = \{ r_T(v)\}$.

For any function $\phi$, set $R_T \phi := \phi \circ r_T$. Observe that
$R_T\phi |_T = \phi |_T$ and that $R_T\phi$ is locally constant outside $T$.

Moreover we have the following
\begin{pro}
Pick any subharmonic function $\phi$
Then for any finite subtree $T$, $R_T\phi$ is subharmonic, $R_T\phi \ge \phi$ and $\Delta (R_T\phi) = (r_T)_* \Delta \phi$.
\end{pro}
\proof
Set $\Delta\phi=\rho$. Then we have $$R_T\phi(w)=\int_{V_{\infty}}g_v(r_T(w))d\rho$$$$=\int_{V_{\infty}}g_{r_T(v)}(w)d\rho=\int_{V_{\infty}}g_{v}(w)dr_{T*}\rho=g_{r_{T*}\rho}$$ which concludes our proposition.
\endproof

Let $T$ be a finite tree containing $\{-\deg\}$ such that for all points $v\in T$, we have $\alpha(v)>-\infty.$
Let $\phi=g_{\rho}$ be a subharmonic function satisfying $\supp \,\rho\subseteq T.$ Set $t(v):=-\alpha(v).$
Let $E$ be the set of all edges of $T$. For each edge $I=[w_1,w_2]\in E$, this function $t(v)$ parameterizes $I$. Denote by $\frac{d^2\phi|_I}{dt^2}dt$
the usual real Laplacian of $\phi|_I$ on the segment $I$ i.e.
the unique measure on $I$ such that
\begin{points}
\item For any segment $(v_1,v_2)\subseteq I$, we have $\int_{[v_1,v_2]}\frac{d^2\phi|_I}{dt^2}dt=D_{\vec{v_1}}\phi(v_1)+D_{\vec{v_2}}\phi(v_2)$
where $\vec{v}_i$ is direction at $v_i$ in $(v_1,v_2)$ for $i=1,2.$
\item $\frac{d^2\phi|_I}{dt^2}dt\{w_i\}=-D_{\vec{w}_i}\phi$ where $\vec{w}_i$ is direction at $w_i$ in $I$ for $i=1,2.$
\end{points}

%

\begin{pro}\label{prosupplapfinittreesub}
We have
\begin{points}
\item
$$\Delta\phi=\phi(-\deg)\delta_{-\deg}+\sum_{I\in E}\frac{d^2\phi|_I}{dt^2}dt;$$
\item
the mass of $\Delta\phi$ at a point $v\in T$ is given
by $\phi(-\deg)\delta_{-\deg}\{x\}+\sum D_{\overrightarrow{v}}\phi$
the sum is over all tangent directions $\overrightarrow{v}$ in $T$ at $v$;
\item
for any segment $I$ contained in $T$,
$\phi|_I$ is convex  and for any point $v\in T$, we have $$\phi(-\deg)\delta_{-\deg}\{v\}+\sum D_{\overrightarrow{v}}\phi\geq 0$$ where $\delta_{-\deg}$ is the dirac measure at $-\deg$ and
the sum is over all tangent directions $\overrightarrow{v}$ in $T$ at $v$.
\end{points}
\end{pro}
\proof[Sketch of the proof]
First check that our proposition holds when $\phi=g_v$ for any $v\in T$. Since all the conclusions in our proposition are linear, they hold for $g_{\rho}(w)=\int_{V_{\infty}}g_v(w)d\rho=\int_{T}g_v(w)d\rho$ also.
\endproof

\begin{thm}\label{thmsubharsequence}
Let $X_n\in \mathcal{C}$, $n\geq 0$ be a sequence of compactifications such that $X_{n+1}$ dominates $X_{n}$ for all $n\geq 0$ and let $T$ be the closure of $\cup_{n\geq 0}\Gamma_{X_n}.$
Suppose that we are given a sequence $\phi_n$ of subharmonic functions satisfying $\Supp \Delta \phi_n\subseteq \Gamma_{X_n}$ and $R_{\Gamma_{X_n}}\phi_{m}=\phi_{n}$ when $m\geq n$.

Then there exists a unique subharmonic function $\phi\in \SH(V_{\infty})$  satisfying $\Supp\Delta\phi\subseteq T$, $R_{\Gamma_{X_n}}\phi=\phi_{n}$ and $\phi=\lim_{n\to\infty}\phi_{n}$.
\end{thm}
\proof Set $\rho_n:=\Delta\phi_{n}$. For any $m\geq n$, we have $r_{X_n}\rho_m=\rho_{n}$. It follows that $\rho_n(V_{\infty})$ is independent on $n$ and we may suppose that $\rho_n(V_{\infty})=1$ for all $n\geq 0.$
Given a continuous function $f$ on $V_{\infty}$ and a real number $\varepsilon>0$, by Lemma \ref{lemaproximationfbyfx}, there exists $N\geq 0$ such that $|f\circ r_{X_n}-f\circ r_{X_m}|\leq \varepsilon$ for all $n,m\geq N.$ It follows that $|\int_{V_{\infty}}fd\rho_n-\int_{V_{\infty}}fd\rho_m|\leq \varepsilon$ for all $n,m\geq N.$
It follows that $\lim_{n\to \infty}\int_{V_{\infty}}fd\rho_n$ exists.

The functional $f\mapsto \lim_{n\to \infty}\int_{V_{\infty}}fd\rho_n$ is continuous, linear and positive, and thus defines a positive Radon measure $\rho.$ Observe that $r_{\Gamma_{X_n}}\rho=\rho_n$ for all $n\geq 0$ and $\rho_n\to\rho$ when $n\to\infty$.
Set $\phi:=g_{\rho}$. We have $R_{\Gamma_{X_n}}\phi=\phi_n$. By Proposition \ref{prosubhapbyrestogx}, we get $\phi=\lim_{n\to\infty}\phi_n$.
\endproof


\subsection{Main properties of subharmonic functions}
The next result collects some properties of subharmonic functions.

\begin{thm}
Pick any subharmonic function $\phi$ on $V_\infty$. Then
\begin{points}
\item $\phi$ is decreasing and $\phi(-\deg)=\Delta\phi(V_{\infty})>0$ if $\phi\neq 0$;
\item $\phi$ is upper semicontinuous;
\item  for any valuation $v\in V_\infty$ the function
$t \mapsto \phi(v_t)$ is convex, where $v_t$ is the unique valuation in $[-\deg , v]$ of skewness $t$.
\end{points}
\end{thm}
\begin{proof}
The first statement follows from Proposition \ref{proshdecreasing} and the equality
$$\phi(-\deg)=\int_{V_{\infty}}g_{v}(-\deg)d\rho(v)=\rho(V_{\infty}).$$

The second statement is a consequence of  Proposition \ref{prosubhapbyrestogx} and Proposition\ref{prosupplapfinittreesub} that impels that $\phi\circ r_X$ is continuous on $V_{\infty}$ for any $X\in \mathcal{C}$. The last statement follows from
Proposition \ref{prosupplapfinittreesub}.
\end{proof}

Now pick any direction $\vec{v}$ at a valuation $v\in V_\infty$.
One may define the directional derivative $D_{\vec{v}} \phi$ of any subharmonic function as follows. If $\alpha(v)\neq -\infty$, pick any map $t \in [0, \epsilon) \mapsto v_t$ such that
$v_0 = v$, $|\alpha (v_t) - \alpha (v_0)| = t$ and $v_t$ determines $\vec{v}$ for all $t>0$.
By property (iii) above, the function $t \mapsto\phi(v_t)$ is convex and continuous at $0$,
so that its right derivative is well-defined. We set
$$
D_{\vec{v}} \phi := \left. \frac{d}{dt}\right|_{t=0} \phi(v_t)~.
$$
This definition does not depend on the choice of map $t\mapsto v_t$. If $\alpha(v)=-\infty$, then $v$ is an endpoint in $V_{\infty}$ and there exists a unique direction $\vec{v}$ at $v$. For any $w<v$, denote by $\vec{w}$ the direction at $w$ determined by $v$. Then we define
$$
D_{\vec{v}} \phi := -\lim_{w\rightarrow v}D_{\vec{w}} \phi
$$
which exists since $\phi|_{[-\deg,v]}$ is convex.

Given any direction $\vec{v}$ at a valuation in $V_\infty$, recall that $U(\vec{v})$ is the open set of valuations determining $\vec{v}$.
\begin{thm}\label{thmdisclap}
Pick any subharmonic function $\phi$ on $V_\infty$. Then
 one has
\begin{equation*}
\Delta \phi (U(\vec{v})) = -D_{\overrightarrow{v}}\phi
\end{equation*}
for any direction $\vec{v}$ that is not determined by $-\deg$.
In particular, one has
\begin{eqnarray*}
\Delta \phi \{ -\deg\} &=& \sum_{\overrightarrow{v}\in \Tan_{-\deg}}D_{\overrightarrow{v}}\phi + \phi(-\deg)~; \text{ and } \\
\Delta \phi \{ v\} &=& \sum_{\overrightarrow{v}\in \Tan_v}D_{\overrightarrow{v}}\phi
\end{eqnarray*}
if $v \neq -\deg$.
\end{thm}
\begin{proof}Since $\vec{v}$ is not determined by $-\deg$, $v$ is not an endpoint of $V_{\infty}.$ Pick $w\in U(\vec{v})$, we have $w>v.$ Set $I:=[-\deg,w]$.
We have $\Delta R_I\phi (U(\vec{v}))=\int_{V_{\infty}}\frac{d^2\phi}{dt^2}dt=\int_{V_{\infty}}d\frac{d\phi}{dt} = -D_{\overrightarrow{v}}R_I\phi$. Since $R_I\phi|_I=\phi|_I$ and $\Delta R_I=r_{I*}\Delta \phi$, we have $\Delta R_I\phi (U(\vec{v}))=\Delta\phi (U(\vec{v}))$ and $D_{\overrightarrow{v}}R_I\phi=D_{\overrightarrow{v}}\phi$. It follows that $\Delta \phi (U(\vec{v})) = -D_{\overrightarrow{v}}\phi$.

If $v=-\deg$, then we have $$\phi(-\deg)=\Delta \phi(V_{\infty})= \Delta \phi\{ -\deg\}+\sum_{\overrightarrow{v}\in \Tan_{-\deg}}\Delta\phi(U(\vec{v}))$$$$=\Delta \phi\{ -\deg\}-\sum_{\overrightarrow{v}\in \Tan_{-\deg}}D_{\overrightarrow{v}}\phi.$$ It follows that $$\Delta \phi \{ -\deg\} = \sum_{\overrightarrow{v}\in \Tan_{-\deg}}D_{\overrightarrow{v}}\phi + \phi(-\deg).$$

If $v\neq -\deg$, let $w_n$ be a sequence of valuations in $[-\deg,v)$. Denote by $\vec{w_n}$ the direction at $w_n$ determined by $v$ and $\vec{v_0}$ the direction at $v$ determined by $-\deg$. Observe that $$-\lim_{n\rightarrow\infty}D_{\vec{w_n}}\phi=\lim_{n\rightarrow\infty}\Delta\phi(U(\vec{w_n}))=\Delta \phi \{ v\}+\sum_{\overrightarrow{v}\in \Tan_v\setminus \{\vec{v_0}\}}\Delta\phi(U(\vec{w_n})).$$ It follows that $D_{\vec{v_0}}\phi=\Delta \phi \{ v\}- \sum_{\overrightarrow{v}\in \Tan_v}D_{\overrightarrow{v}}\phi$ and then $$\Delta \phi \{ v\} = \sum_{\overrightarrow{v}\in \Tan_v\setminus \{\vec{v_0}\}}D_{\overrightarrow{v}}\phi.$$
\end{proof}
\begin{thm}\label{thmcheckafunctionsubhar}
Suppose $\phi: V_\infty \to [-\infty, +\infty)$ is a function such that
\begin{points}
\item  for any valuation $v\in V_\infty$ the function
$[\alpha(v), 1] \ni t \mapsto \phi(v_t)$ is continuous and convex,
where $v_t$ is the unique valuation in $[-\deg , v]$ of skewness $t$;
\item the inequalities
\begin{equation}\label{eqcharsh}
\sum_{\vec{v}\in \Tan_{-\deg}}D_{\vec{v}}\phi + \phi(-\deg) \geq 0~; \text{ and }
\sum_{\vec{v}\in \Tan_{v}}D_{\vec{v}}\phi \geq 0
\end{equation}
 are  satisfied for all valuations $v\neq -\deg$.
\end{points}
Then $\phi$ is subharmonic.
\end{thm}

\begin{proof}
Let $v_1,v_2\in V_{\infty}$ be two valuations satisfying $v_1<v_2.$ There exists an end point $w\in V_{\infty}$ satisfying $v_1,v_2\in [-\deg,w]$. Denote by $\vec{w}$ the unique direction in $\Tan_w$. By (ii), we have $D_{\vec{w}}\phi\geq 0.$ Since $\phi$ is convex on $[-\deg,w]$, it is decreasing on $[-\deg,w]$. It follows that $\phi(v_1)\geq \phi(v_2)$ and then $\phi$ is decreasing.

For any $v\in V_{\infty}\setminus\{-\deg\}$, denote by $\vec{v}$ the direction at $v$ determined by $-\deg.$ For any $n\geq 1$, set $T_n:=\{v\in V_{\infty}\setminus \{-\deg\}|\,\,D_{\vec{v}}\phi\geq 1/n \}.$ Since the map $v\mapsto D_{\vec{v}}\phi$ is non negative and decreasing, it follows that $T_n$ is a tree.

We claim that $T_n$ is a finite tree.
If $T_n=\{-\deg\}$, there is nothing to prove.

For convenience, we define  $D_{\overrightarrow{-\deg}}\phi:=-\sum_{\vec{v}\in \Tan_{-\deg}}D_{\vec{v}}\phi=\phi(-\deg).$
Let $w$ be a valuation in $T_n$ and $v_1,\cdots, v_m$ be valuations in $T_n$ satisfying $v_i\wedge v_j=w$ for all $i\neq j.$ Denote by $\vec{w_i}$ the direction at $w$ determined by $v_i.$
Then we have $$\sum_{i=1}^mD_{\vec{v_i}}\phi\leq \sum_{i=1}^m-D_{\vec{w}_i}\phi\leq D_{\vec{w}}\phi.$$

Pick $m$ valuations $v_1,\cdots,v_m\in T_n$ such that any two valuations $v_i$, $v_j$ $i\neq j$ are not comparable. Let $S$ be the set of maximal elements in the set $\{v_i\wedge v_j|\,\, 1\leq i<j\leq m\}$ and write $S=\{w_1,\cdots,w_l\}$. Observe that $l\leq m-1$ if $m\geq 2.$ Let $S_w$ be the set of $v_i$ satisfying $v_i>w$. Then we have $\sum_{v\in S_w}D_{\vec{v}}\phi\leq D_{\vec{w}}\phi$ and $\{v_1,\cdots,v_m\}=\coprod_{w\in S}S_w.$ It follows that $\sum_{i=1}^mD_{\vec{v_i}}\phi\leq\sum_{w\in S}D_{\vec{w}}\phi.$ By induction, we have
$$\sum_{i=1}^mD_{\vec{v_i}}\phi\leq D_{\vec{\wedge_{i=1}^mv_i}}\phi\leq D_{\overrightarrow{-\deg}}\phi=\phi(-\deg) .$$ Since $D_{\vec{v_i}}\phi\geq 1/n$, we conclude that $m\leq n\phi(-\deg).$ This fact implies that $T_n$ is a finite tree with at most $n\phi(-\deg)$ end points.

As in the proof of Lemma \ref{lem714favvalu}, we an now show that there exists a sequence of admissible compactification $X_n\in \mathcal{C}$, $n\geq 0$ such that $X_{n+1}$ dominates $X_{n}$ for all $n\geq 0$ and $\cup_{n\geq 0}T_n$ is contained in the closure of $\cup_{n\geq 0}\Gamma_{X_n}.$ Set $\phi_n:=R_{\Gamma_{X_n}}\phi.$

Let $v$ be a point in $V_{\infty}$. Set $I:=[-\deg,v]$ and $I_n:=I\cap \Gamma_{X_n}=[-\deg,v_n]$. Observe that $v_n$ is increasing and define $v':=\lim_{n\to\infty}v_n$. Observe that for all $(v',v]\subseteq V_{\infty}\setminus (\cup_{n\geq 1}T_n)$,  and then $D_{\vec{w}}=0$ for all $w\in (v',v]$.
It follows that $$\phi(v)=\phi(w)=\lim_{n\to\infty}\phi(v_n)=\lim_{n\to\infty}\phi_n(v).$$

Denote by $\rho_n:= \phi_n(-\deg)\delta_{-\deg}\{x\}+\sum\frac{d^2\phi|_I}{dt^2}dt$ where the sum is over all edges of $\Gamma_{X_n}.$
It is a Radon measure supported on $\Gamma_{X_n}$.
It follows that $\phi_n=g_{\rho_n}$ which is subharmonic and $\phi_n=R_{\Gamma_{X_n}}\phi_m$ for any $m\geq n$.  Then we conclude by applying Theorem \ref{thmsubharsequence}.
\end{proof}

\smallskip

 The next result collects the main properties of the space of subharmonic functions.
\begin{thm}\label{thmconstrushf}
The sets $\SH(V_\infty)$ and $\SHP$  are convex cones that are stable by $\max$. In other words, given any $c>0$, and any $\phi, \phi' \in \SH(V_\infty)$ (resp. in $\SHP$), then
$c\phi, \phi + \phi'$ and $\max \{ \phi, \phi'\}$ all belong to $\SH(V_\infty)$ (resp. to $\SHP$).
\end{thm}
\begin{proof}By Theorem \ref{thmcheckafunctionsubhar}, it is easy to check that $c\phi$ and  $\phi + \phi'$  all belong to $\SH(V_\infty)$ (resp. to $\SHP$) when $c>0$, and $\phi, \phi' \in \SH(V_\infty)$ (resp. in $\SHP$).

We only have to check that $\max \{ \phi, \phi'\}$ belongs to $\SH(V_{\infty})$ when $\phi, \phi' \in \SH(V_\infty)$. It is easy to see that the condition (i) in Theorem \ref{thmcheckafunctionsubhar} holds. For any point $v\in V_{\infty}$ and any direction $\vec{v}$ at $v$, if $\phi(v)> \phi'(v)$ (resp. $\phi(v)< \phi'(v)$), then $D_{\vec{v}}\max \{ \phi, \phi'\}=D_{\vec{v}}\phi$ (resp. $D_{\vec{v}}\max \{ \phi, \phi'\}=D_{\vec{v}}\phi'$). It follows that the condition (ii) in Theorem \ref{thmcheckafunctionsubhar} holds when $\phi(v)\neq \phi'(v)$. Otherwise, if $\phi(v)= \phi'(v)$, we have $D_{\vec{v}}\max \{ \phi, \phi'\}=\max\{D_{\vec{v}}\phi,D_{\vec{v}}\phi'\}$ and then the condition (ii) in Theorem \ref{thmcheckafunctionsubhar} holds. Now we conclude by applying Theorem \ref{thmcheckafunctionsubhar}.
\end{proof}

\subsection{Examples of subharmonic functions}
For any nonconstant polynomial $Q\in k[x,y]$, we define the function
$$\log |Q| (v) :=  -v(Q)~,$$
which takes values in $[-\infty,\infty)$. 
\begin{pro}
The function $\log|Q|$ is subharmonic, and
$$
\Delta (\log |Q|) = \sum_i m_i \delta_{v_{s_i}}
$$
where $s_i$ are the branches of the curve $\{Q=0\}$ at infinity,
and $m_i$ is the intersection number of $s_i$ with the line at infinity in $\P^2_k$.
\end{pro}

\begin{proof}[Sketch of proof]
Let $g = \sum_i m_i g_{v_{s_i}}$. One has to prove that $\log|Q| = g$. To that end, we pick any
admissible compactification $X$ of $\A^2_k$ and prove that $\log|Q|(v_E) = g(v_E)$ for any irreducible component of $X_\infty := X\setminus \A^2_k$. The proof then goes by induction on the number of irreducible component of $X_\infty$ and observing that this number is $1$ only if $X = \P^2_k$.
\end{proof}

\begin{pro}\label{proshpos}
The function $\log^+|Q| := \max \{ 0, \log|Q|\}$ belongs to $\SHP$.

Denote by  $s_1,\cdots,s_l$ the branches of $\{Q=0\}$ at infinity and by $T$ the convex hull of  $\{ - \deg, v_{s_1},\cdots, v_{s_l}\}$.
Then the support of $\Delta (\log^+|Q|)$ is the set of points $v\in T$ satisfying $v(Q)=0$ and $w(Q)<0$ for all $w\in (v,-\deg]$.
\end{pro}
In particular,  $\Supp \Delta (\log^+|Q|)$ is finite.
\begin{proof} By Theorem \ref{thmconstrushf} we have $\log^+|Q| \in \SH(V_{\infty})$.
Observe that $\log^+|Q|$ is locally constat on $V_{\infty}\setminus T$ so that the support of $\Delta \log^+|Q|$ is included in $T$.
Let $\{v_1,\cdots,v_m\}$ be the set of points $v\in T$ satisfying $v(Q)=0$ and $w(Q)<0$ for all $w\in (v,-\deg]$.
For any $v\in V_{\infty}$, we have $\log|Q|\geq \deg(Q)\alpha(v)$. It follows that $\alpha(v_i)\leq 0$ and then $v_i\neq -\deg$.
Denote by $m_i's$ the intersection number of $s_i$ with the line at infinity in $\mathbb{P}^2_k$.
For any $i=1,\cdots,m$, denote by $S_i$ the set of branches of the curve $s_j$ satisfying $v_{s_j}>v_i$.
Observe that $S_i\neq \emptyset$ and $\{s_1,\cdots,s_l\}=\coprod_{i=1}^mS_i$.
By Theorem \ref{thmdisclap}, we have $\Delta\log^+|Q|\{v_i\}=\sum_{s_j\in S_i}m_j>0$.
Then we have $\sum_{i=1}^m\Delta\log^+|Q|\{v_i\}=\sum_{j=1}m_j=\deg(Q)=\log^+|Q|(-\deg)=\Delta(\log^+|Q|)(V_{\infty})$.
It follows that $$\Delta(\log^+|Q|)=\sum_{i=1}^m(\sum_{s_j\in S_i}m_j)\delta_{v_i}.$$
It follows that $\supp \Delta (\log^+|Q|)=\{v_1,\cdots,v_m\}$ and moreover we have $m\leq \deg(Q)$.
\end{proof}
\subsection{The Dirichlet pairing}
 Let $\phi, \psi$ be any two subharmonic functions on $V_\infty$.
Since $\phi$ is bounded from above one can define the Dirichlet pairing
$$\langle\phi,\psi\rangle:=\int_{V_{\infty}^2}\alpha(v\wedge w)\Delta\phi(v)\Delta\psi(w)\in [-\infty,+\infty).$$
Observe that $\langle\phi,\psi\rangle=\langle\psi,\phi\rangle$.
\begin{pro}The Dirichlet pairing induces a symmetric bilinear form on $\SH(V_{\infty})$ that satisfies
$$\langle \phi, \psi \rangle= \int_{V_\infty} \phi \,\Delta \psi~\eqno (*).$$
\end{pro}
\proof The linearity and the symmetry are obvious from the definition. Equation (*) follows from  Fubini's Theorem.
\endproof

We shall prove
\begin{thm}[Hodge inequality]\label{thmpropofdirpair}
For any two subharmonic functions  $\phi,\psi$, we have
$$(\phi(-\deg)\psi(-\deg)-\langle\phi,\psi\rangle)^2\leq (\phi(-\deg)^2-\langle\phi,\phi\rangle)(\psi(-\deg)^2-\langle\psi,\psi\rangle).$$
\end{thm}
\begin{proof}[Proof of the Theorem~\ref{thmpropofdirpair}]
We first need the following
\begin{pro}\label{prohodgeindexchangorder} Let $\phi,\psi$ be two subharmonic functions in $\SH(V_{\infty})$.
Then there exists a sequence of compactifications $X_n\in \mathcal{C}$, $n\geq 0$ such that $X_{n+1}$ dominates $X_{n}$ for $n\geq 0$ and
$\langle\phi,\psi\rangle=\lim_{n\rightarrow\infty}\langle R_{\Gamma_{X_n}}\phi,R_{\Gamma_{X_n}}\psi\rangle.$
\end{pro}
W only have to prove our theorem in the case
 $\Delta\phi$ and $\Delta\psi$ are supported on a finite subtree $T$ of $V_{\infty}$. Set $t(v):=-\alpha(v)$ for $v\in T.$
Denote by $E$ the set of all edges of $T$, $v_1^I,v_2^I$ the two endpoints of $I$ and $\vec{v}_{1}^I, \vec{v}_{2}^I$ the two direction at $v_1^I$ and $v_2^I$. Denote by $\{v_1,\cdots,v_l\}$ the set of all endpoints and branch points in $T$ and $T_v$ the set of direction at $v$ in $T.$

By integration by parts, we have
$$\int_{I}\phi\frac{d^2\psi}{dt^2}=-\int_{I}\frac{d\phi}{dt}\frac{d\psi}{dt}dt$$ for all $I\in E.$
Then we have
$$\langle\phi,\psi\rangle=\int_{\V_{\infty}}\phi(v)\psi(-\deg)\delta_{-\deg}(v)+\sum_{I\in E}\int_{I}\phi\frac{d^2\psi}{dt^2}
=\phi(-\deg)\psi(-\deg)-\int_{T}\frac{d\phi}{dt}\frac{d\psi}{dt}dt.$$
It follows that $\langle\phi,\psi\rangle=\langle\psi,\phi\rangle$,
and by Cauchy inequality, we get $$(\phi(-\deg)\psi(-\deg)-\langle\phi,\psi\rangle)^2\leq (\phi(-\deg)^2-\langle\phi,\phi\rangle)(\psi(-\deg)^2-\langle\psi,\psi\rangle).$$
\end{proof}

%

\proof[Proof of Proposition \ref{prohodgeindexchangorder}]By Proposition \ref{prosubhapbyrestogx}, there exists a sequence of compactifications $X_n\in \mathcal{C}$ $n\geq 0$ such that $X_{n+1}$ dominates $X_{n}$ for $n\geq 0$ and $R_{\Gamma_{X_n}}\phi$ (resp. $R_{\Gamma_{X_n}}\psi$) decreases pointwise to $\phi$ (resp. $\psi$).

We have
$$|\langle\phi,\psi\rangle-\langle R_{\Gamma_{X_n}}\phi,R_{\Gamma_{X_n}}\psi\rangle|\leq \Big|\int_{V_{\infty}}R_{\Gamma_{X_n}}(\phi) \Delta R_{\Gamma_{X_n}}(\psi)-\int_{V_{\infty}}\phi \Delta R_{\Gamma_{X_n}}\psi\Big|$$$$+\Big|\int_{V_{\infty}}\phi \Delta R_{\Gamma_{X_n}}\psi-\int_{V_{\infty}}\phi\Delta \psi\Big|.$$
Observe that $$\Big|\int_{V_{\infty}}R_{\Gamma_{X_n}}(\phi) \Delta R_{\Gamma_{X_n}}(\psi)-\int_{V_{\infty}}\phi \Delta R_{\Gamma_{X_n}}\psi\Big|=0$$ and $$\Big|\int_{V_{\infty}}\phi \Delta R_{\Gamma_{X_n}}\psi-\int_{V_{\infty}}\phi\Delta \psi\Big|\to 0$$  by monotone convergence.
It follows that $$|\langle\phi,\psi\rangle-\langle R_{\Gamma_{X_n}}\phi,R_{\Gamma_{X_n}}\psi\rangle|\to\infty$$ as $n\to\infty$.

\endproof

Finally, we collect two useful results.
\begin{pro}\label{prortgephipsi}
Pick any two subharmonic functions $\phi, \psi \in \SH(V_\infty)$. For any finite subtree $T\subset V_\infty$ one has
$$
\langle R_T\phi, R_T\psi \rangle \ge \langle \phi, \psi \rangle~
.$$
\end{pro}
\begin{proof}
%
%
%
%
%
Since $R_T\phi\geq \phi$, for any $\psi\in \SH(V_{\infty})$ we have $\langle R_T\phi, \psi \rangle=\int_{V_{\infty}}R_T\phi\Delta\psi \ge\int_{V_{\infty}}\phi\Delta\psi=\langle \phi, \psi \rangle$.  It follows that $$
\langle R_T\phi, R_T\psi \rangle \ge \langle \phi, R_T\psi \rangle\geq \langle \phi, \psi \rangle.
$$
\end{proof}
\begin{pro}\label{prophiphi}Pick any subharmonic function $\phi \in \SH(V_\infty)$. For any finite subtree $T\subset V_\infty$ one has
$$
\langle R_T\phi, R_T\phi \rangle \ge \langle \phi, \phi \rangle~
$$ and the equality holds if and only if $\Delta\phi$ is supported on $T$.
\end{pro}
\proof By Proposition \ref{prortgephipsi}, we only have to show that $\langle R_T\phi, R_T\phi \rangle > \langle \phi, \phi \rangle$ when $\Delta\phi$ is not supported on $T$.

Suppose that $\Delta\phi$ is not supported on $T$. It follows that $\Delta\phi(V_{\infty}\setminus T)>0$. Pick $X\in \mathcal{C}$ such that $r_{X*}\Delta\phi(V_{\infty}\setminus T)>0$, and set $Y:=T\cup \Gamma_X$, so that $Y$ is a finite tree.

Since $\langle R_T(\phi), R_T(\phi) \rangle\geq\langle R_Y(\phi), R_Y(\phi) \rangle\geq\langle \phi, \phi \rangle$, by replacing $\phi$ by $R_Y\phi$, we may suppose that $\Delta\phi$ is supposed by $Y$. There exists a connected component $U$ of $Y\setminus T$ satisfying $\int_U\Delta \phi>0$. There exists a unique point $y_0\in \overline{U}\cap T$ where $\overline{U}$ is the closure of $U$ in $Y$.  It follows that $\phi(y)<\phi(y_0)=R_T\phi(y)$ for all $y\in U.$ Then we conclude that
$$\langle \phi, \phi \rangle=\int_{Y}\phi\Delta \phi=\int_{T\setminus U}\phi\Delta \phi+\int_{U}\phi\Delta \phi$$$$<\int_{T\setminus U}\phi\Delta \phi+\int_{U}R_T\phi\Delta \phi\leq \int_{T\setminus U}R_T\phi\Delta \phi+\int_{U}R_T\phi\Delta \phi$$$$=\int_{Y}R_T\phi\Delta \phi=\int_{Y}\phi\Delta R_T(\phi)$$$$=\int_{Y}R_T(\phi)\Delta R_T(\phi)=\langle R_T(\phi), R_T(\phi) \rangle.$$
\endproof

\subsection{Positive subharmonic functions}
We prove here a technical result that will play an important role in the next section.

%
%

For any set $S\subset V_\infty$ we define $B(S) := \cup_{v\in S} \{ w, \, w\ge v\}$.

\pro\label{proapoxishplus}Let $\phi$ be a function in $\SHP$ such that $\langle\phi,\phi\rangle=0$ and $\Supp\Delta\phi=\{v_1,\cdots,v_s\}$ where $s$ is a positive integer.

Then for any finite set $S\subseteq B(\{v_1,\cdots,v_s\})$ satisfying $\{v_1,\cdots,v_s\}\not\subseteq S$, there exists a function $\psi\in \SHP$ such that
\begin{itemize}
\item $\psi(v)=0$ for all $v\in B(S)$;
\item
 $\langle\psi,\psi\rangle>0.$
\end{itemize}
\endpro

\exe Let $Q\in k[x,y]$ be any nonconstant polynomial. Proposition~\ref{proshpos}
implies that
$\log^+|Q|\in \SHP$, $\langle \log^+|Q|,\log^+|Q|\rangle=0$ and $\#\Supp\Delta \log^+|Q|<\infty$
so that the preceding proposition applies to $\phi = \log^+|Q|$.

\proof
Write $\Delta \phi = \sum_{i=1}^s r_i \delta_{v_i}$ with $r_i>0$.
Since $\langle\phi,\phi\rangle=0$ we have $\phi(v_i) =0$ for all $i$.
Observe now that the restriction of $\phi$ to any segment $[-\deg , v_i]$ is
not  locally constant. It follows that the sets $B(\{v_i\})$ are  disjoint, or in other
words that $v_i \wedge v_j < v_i$ for any $i \neq j$.

\smallskip

Suppose first that there exists an index $i\in \{1,\cdots,s\}$ such that $S\cap B(\{v_i\})=\emptyset$, and denote by $T$ the convex hull of $\{-\deg, v_1,\cdots,v_s\}\setminus \{v_i\}$. Then  $\psi :=R_T\phi$  satisfies all the required conditions.

\smallskip

Otherwise we may suppose that $v_1 \notin S$ and pick $w_1 \in S$ satisfying $w_1 >v_1$.


Choose any $v_1'<v_1$ such that $(\supp \Delta \phi)\cap B(\{v_1'\})=\{v_1\}$, and $w^1\in (v_1',v_1)$, $w^2\in (v_1,w_1)$ such that $\alpha(w^1)-\alpha(v_1)=\alpha(v_1)-\alpha(w^2).$ The subharmonic function $\psi:=\sum_{i=2}^sr_ig_{v_i}+\frac{r_1}2(g_{w^1}+g_{w^2})$
satisfies all required conditions.
\endproof

\subsection{The class of $\mathbb{L}^2$ functions}
We define $\mathbb{L}^2(V_{\infty})$ to be the set of functions $\phi:\{v\in V_{\infty}|\,\,\alpha(v)>-\infty\}\to \mathbb{R}$ such that
$\phi=\phi_1-\phi_2$ on $\{v\in V_{\infty}|\,\,\alpha(v)>-\infty\}$ with
$\phi_i\in \SH (V_\infty)$ and $\langle\phi_i,\phi_i\rangle>-\infty$ for $i=1,2$.
Then $\mathbb{L}^2(V_{\infty})$ is a vector space.

For sake of convenience, we shall always extend $\phi$ to $V_{\infty}$ by setting $\phi(v)$ to be an arbitrary number in
$\phi(v)\in [\liminf_{w<v}\phi(w),\limsup_{w<v}\phi(w)]$ when $\alpha(v)=-\infty$.

Observe that by Proposition~\ref{prohodgeindexchangorder} (iii), we have $\langle\phi_1,\phi_2\rangle>-\infty $ so that the pairing $\langle \cdot,\cdot\rangle$
extends to $\mathbb{L}^2(V_{\infty})$  as a symmetric bilinear form and  Hodge inequality \ref{thmpropofdirpair} is still valid.

%
%

All bounded subharmonic functions are contained in $\mathbb{L}^2(V_{\infty})$.
In particular, $g _v\in \mathbb{L}^2(V_{\infty})$ if $\alpha(v)>-\infty$ and $\SHP\subseteq \mathbb{L}^2(V_{\infty})$.
\newpage


\section{Proof of the Main Theorem}\label{sectionpolytakingnova}
\subsection{First reductions}
%
%
Let us recall the setting from the introduction. Let $R:=k[x,y]$ denote the ring of polynomials in two variables over an algebraically closed field $k$.
Let $S$ be a finite set of valuations on $R$ that are trivial on $k$. We define
$R_S = \cap_{v \in S} \{ P \in R, \, v(P) \ge 0\}$.
This is a $k$-subalgebra of $k[x,y]$ and we denote by
$\delta(S)\in \{0,1,2\}$ the transcendence degree of its field of fraction over $k$.

We first do the following reduction.
\begin{lem}\label{lemreducetosubsetofvin}Given any finite set of valuations $S$ on $R$ that are trivial on $k$ and centered at infinity,
we have $\delta(S) =2$  if and only if $\delta(\bar{S}) =2$.
\end{lem}
\proof
Since $R_S \subset R_{\bar{S}}$ it follows that $\delta(S)=2$ implies $\delta(\bar{S})=2$.

Conversely suppose that $\delta(\bar{S})=2$.
Let $v_1, \ldots , v_s$ be the rank $2$ valuations in $S$ whose associated valuations
$\bar{v}_1, \ldots , \bar{v}_s$
in $V_\infty$ are divisorial. Observe that when $v\in S \setminus \{ v_1, \ldots, v_s\}$ then
$R_{\{v\}}=R_{\{\bar{v}\}}$.

By Theorem \ref{thmsuvsectionprorefshr} (ii), there is a nonzero polynomial $P\in R$
such that $v(P) >0$ for all $v \in \bar{S}$. Pick any polynomial $Q$.
Then for $m$ large enough, we have $v(P^mQ) >0$ for all $v\in \bar{S}$. In particular, we get
$\bar{v}_i(P^mQ) >0$ which implies $v_i(P^mQ)>0$. We conclude that
$P^mQ$ also belongs to $R_S$ so that the fraction field of $R_S$ is equal to $k(x,y)$ and
$\delta(S)=2$.
\endproof

\bigskip

In the rest of this section, let $S\subset V_\infty$ be a finite set. It will be convenient to use
the following terminology.
\begin{defi}
A subset of valuations $S \subset V_\infty$ is said to be \emph{rich} when $\delta(S) = 2$.
\end{defi}
We shall also write:
\begin{itemize}
\item
$S^{\min}\subset S$ for the set of valuations that are minimal for the order relation restricted to $S$;
\item
$S_{+} \subset S$ for the subset of valuations in $S$ with finite skewness;
\item
$S^{\min}_{+} \subset S^{\min}$ for the subset of valuations in $S^{\min}$ with finite skewness;
\item
$B(S)$ for the set of all valuations $v\in V_\infty$ such that $v \ge w$ for some $w \in S$;
\item
$B(S)^{\circ}$ for the interior of $B(S)$;
\item
$M(S)$ for the symmetric matrix whose entries are given by $[\alpha(v_i \wedge v_j)]_{1 \le i,j \le l}$.
\end{itemize}
The set $B(S)$ is compact and has as many connected components as there are elements of $S^{\min}$.
In fact, the boundary of any connected component of $B(S)$ is a singleton, and this point lies in $S^{\min}$. Observe that $R_{S^{\min}} = R_S$.

\medskip

The next result follows directly from Hodge index theorem in the case of
divisorial valuations and
by a continuity argument in the general case.
\begin{lem}\label{lemhodgeindextheoremvalu}Let $S$ be a finite subset of $V_{\infty}$ such that $\alpha(v)>-\infty$ for all $v\in S$.
Then the  symmetric matrix $M(S)$ has at most one non-negative eigenvalue.
\end{lem}

\begin{defi}Let $S$ be a finite subset of $V_{\infty}$. The symmetric matrix $M(S)$ is said to be
negative definite if and only if the matrix $[(\max\{\alpha(v_i\wedge v_j),-t\}]_{1\leq i,j\leq m}$ is negative
define for $t$ large enough.
\end{defi}


Observe that
for $t$ large enough the function $t\mapsto \det(\max\{\alpha(v_i\wedge v_j),t\})_{1\leq i,j\leq l}$ is a polynomial, and that we defined
$$\chi(S) = \lim_{t\to -\infty} (-1)^{\# S}
\det(\max\{\alpha(v_i\wedge v_j),t\})_{1\leq i,j\leq l} \in \R \cup \{\pm \infty\}$$
with the
convention
$\chi(\emptyset):=1$.
When $S = S_{+}$ we simply have
$\chi(S):=(-1)^{\# S}\det((\alpha(v_i\wedge v_j))_{1\leq i,j\leq l})$.

With this definition, lemma \ref{lemhodgeindextheoremvalu}
implies immediately
\begin{lem}\label{lemreducetodet}Let $S$ be a finite subset of $V_{\infty}$. The symmetric matrix $M(S)$ is negative definite if and only if $\chi(S)>0$.
\end{lem}

Finally we make the following reduction
\begin{lem}\label{lemredtosminplus}Let $S$ be a finite subset of $V_{\infty}$. We have $\chi(S)>0$ if and only if $\chi(S^{\min}_{+})>0$.
\end{lem}
\proof
Suppose that $S=\{v_1,\cdots,v_l\}$ and $S_+=\{v_1,\cdots,v_{l'}\}$ where $l'\leq l$.
When $t$ large enough the function $t\mapsto \det(\max\{\alpha(v_i\wedge v_j),t\})_{1\leq i,j\leq l}$ is a polynomial with leading term $\chi(S_{+})t^{l-l'}$. It follows that $\chi(S)>0$ if and only if $\chi(S_{+})>0$.
Now, we may suppose that $S=S_+.$

\smallskip

Since $S^{\min}$ is a subset of $S$, if $M(S)$ is negative definite then $M(S^{\min})$ is negative definite. By Lemma \ref{lemreducetodet}, we conclude the " only if" part.

\smallskip
To prove the " if" part, we suppose that $\chi(S^{\min})>0$.
For any $w\in S^{\min}$, set $S_w:=\{v\in S|\,\,v\geq w\}$. It follows that $S=\coprod_{w\in S^{\min}}S_w$. For any $w\in S^{\min}$, denote by $C(S_w)$ the set of valuations taking forms $\wedge_{v\in S_w'}v$ where $S_w'$ is a subset of $S_w$. Set $C(S):=\coprod_{w\in S^{\min}}C(S_w)$.
We complete the proof of our theorem by induction on the number $\#C(S)-\#S^{\min}$.

If $\#C(S)-\#S^{\min}=0$, then $S=C(S)=S^{\min}$. Our theorem trivially holds.

If $\#C(S)-\#S^{\min}\geq 1$, there exists $w\in S^{\min}$ satisfying $C(S_w)\geq 2.$ Let $w_0$ be a maximal element in $C(S_w)$ then $w_0>w.$ Let $w_1$ be the maximal element in $[w,w_0)\cap S_w$ and set $S_1:=C(S)\setminus \{w_0\}.$ For any valuation $v\in C(S)\setminus \{w_0\}$, we have $v\wedge w_0=v\wedge w_1$.
Then we have
\begin{multline*}
M(C(S))=\left(\begin{matrix}
\alpha(w_0) & \ldots & \alpha(w_0\wedge v) & \ldots & \alpha(w_0\wedge w_1)& \ldots\\
\ldots & \ldots & \ldots & \ldots & \ldots & \ldots  \\
\alpha(v\wedge w_0) & \ldots & \alpha(v) & \ldots & \alpha(v\wedge w_1)& \ldots\\
\ldots & \ldots & \ldots & \ldots & \ldots & \ldots  \\
 \alpha(w_1\wedge w_0) & \ldots & \alpha(w_1\wedge v) & \ldots & \alpha(w_1)& \ldots\\
\ldots & \ldots & \ldots & \ldots & \ldots & \ldots  \\
 \end{matrix} \right)\\
 =\left(\begin{matrix}
\alpha(w_0) & \ldots & \alpha(w_1\wedge v) & \ldots & \alpha(w_1)& \ldots\\
\ldots & \ldots & \ldots & \ldots & \ldots & \ldots  \\
\alpha(v\wedge w_1) & \ldots & \alpha(v) & \ldots & \alpha(v\wedge w_1)& \ldots\\
\ldots & \ldots & \ldots & \ldots & \ldots & \ldots  \\
 \alpha(w_1) & \ldots & \alpha(w_1\wedge v) & \ldots & \alpha(w_1)& \ldots\\
\ldots & \ldots & \ldots & \ldots & \ldots & \ldots  \\
 \end{matrix} \right).
\end{multline*}
It follows that
\begin{multline*}
M(C(S))=\left(\begin{matrix}
\alpha(w_0) & \ldots & \alpha(w_1\wedge v) & \ldots & \alpha(w_1)& \ldots\\
\ldots & \ldots & \ldots & \ldots & \ldots & \ldots  \\
\alpha(v\wedge w_1) & \ldots & \alpha(v) & \ldots & \alpha(v\wedge w_1)& \ldots\\
\ldots & \ldots & \ldots & \ldots & \ldots & \ldots  \\
 \alpha(w_1) & \ldots & \alpha(w_1\wedge v) & \ldots & \alpha(w_1)& \ldots\\
\ldots & \ldots & \ldots & \ldots & \ldots & \ldots  \\
 \end{matrix} \right)\\
 =\left(\begin{matrix}
1 & \ldots & 0 & \ldots & 0& \ldots\\
\ldots & \ldots & \ldots & \ldots & \ldots & \ldots  \\
0 & \ldots & 1 & \ldots & 0& \ldots\\
\ldots & \ldots & \ldots & \ldots & \ldots & \ldots  \\
1 & \ldots & 0 & \ldots & 1& \ldots\\
\ldots & \ldots & \ldots & \ldots & \ldots & \ldots  \\
 \end{matrix} \right)
 \left(\begin{matrix}
\alpha(w_0)-\alpha(w_1) & \ldots & 0 & \ldots & 0& \ldots\\
\ldots & \ldots & \ldots & \ldots & \ldots & \ldots  \\
0 & \ldots & \alpha(v) & \ldots & \alpha(v\wedge w_1)& \ldots\\
\ldots & \ldots & \ldots & \ldots & \ldots & \ldots  \\
 0 & \ldots & \alpha(w_1\wedge v) & \ldots & \alpha(w_1)& \ldots\\
\ldots & \ldots & \ldots & \ldots & \ldots & \ldots  \\
 \end{matrix} \right)\\
 \left(\begin{matrix}
1 & \ldots & 0 & \ldots & 1& \ldots\\
\ldots & \ldots & \ldots & \ldots & \ldots & \ldots  \\
0 & \ldots & 1 & \ldots & 0& \ldots\\
\ldots & \ldots & \ldots & \ldots & \ldots & \ldots  \\
 0 & \ldots & 0 & \ldots & 1& \ldots\\
\ldots & \ldots & \ldots & \ldots & \ldots & \ldots  \\
 \end{matrix} \right).
\end{multline*}
It follows that $\chi(C(S))=\left(\alpha(w_1)-\alpha(w_0)\right)\chi(S_1).$ Since $C(S_1)=S_1=C(S)\setminus \{w_0\}$ and $S_1^{\min}=S^{\min}$,
we have $\chi(S_1)>0$ by induction hypotheses.
Since $\alpha(w_1)-\alpha(w_0)>0$, we have $\chi(C(S))>0$ and $M(C(S))$ is negative definite. Since $M(S)$ is a principal submatrix of $M(C(S))$,
it is also negative definite. It follows that $\chi(S)>0.$
\endproof


%


\subsection{Characterization of rich sets using potential theory on $V_\infty$}

As an important intermediate step towards our Main Theorem we shall prove the following characterization of rich subsets of $V_\infty$ in terms of the existence of
adapted functions in $\mathbb{L}^2(V_{\infty})$.

\begin{thm}\label{prodimtwothengeqz}\label{thmsuvsectionprorefshr}\label{thmsuvsectionprorefshr}Let $S$ be a finite set of valuations in $V_{\infty}$. Then the following statements are equivalent.
\begin{itemize}
\item[(i)] The set $S$ is rich, i.e. $\delta(S) =2$.
\item[(ii)] There exists a nonzero polynomial $P\in R_{S}$ such that $v(P)>0$ for all $v\in S.$
\item[(iii)] There exists a valuation $v\in S$ and a nonzero polynomial $P\in R_{S}$ such that $v(P)>0.$
\item[(iv)] There exists a function $\phi\in \SHP$ such that $\phi(v)=0$ for all $v\in B(S)$ and $\langle\phi,\phi\rangle>0.$
\item[(v)] There exists a function $\phi\in \mathbb{L}^2(V_{\infty})$ such that $\phi(v)=0$ for all $v\in B(S)$ and $\langle\phi,\phi\rangle>0$.
\item[(vi)] There exists a finite set $S'\subseteq V_{\infty}$ such that $S\subseteq B(S')^{\circ}$ and $S'$ is rich.
\end{itemize}
Moreover when these conditions are satisfied, then the fraction field of $R_S$ is equal to $k(x,y)$.
\end{thm}

\proof
Observe first that when (ii) is satisfied, then for any polynomial $Q$ there exists an integer $n$ such that $Q P^n$ belongs to $R_S$. This implies that $k[x,y]$ is included in the fraction field of $R_S$ hence the latter is equal to $k(x,y)$.

We now prove the equivalence between the six statements.
The three implications (ii)$\Rightarrow$(iii), (iv)$\Rightarrow$(v) and (vi)$\Rightarrow$(i)
are immediate.

\smallskip

(i)$\Rightarrow$(ii).
Replacing $S$ by $S^{\min}$, we may suppose that $S=S^{\min}.$
By contradiction, we suppose that $v(P)=0$ for all $v\in S$ and all $P\in R_S\setminus \{0\}.$

For every $v\in S$, we have $\min\{v(x),v(y)\}=-1$.  Since $k$ is infinite, for a general linear polynomial $Q\in k[x,y]$, we have $v(Q)<0$ for all $v\in S$.
Since the transcendence degree of $\Frac(R_S)$ over $k$ is $2$, we have $$\sum_{i=0}^ma_iQ^i=0$$ where $m\geq 1$, $a_i\in R_S.$ We may suppose that $a_m\neq 0$. Let $v$ be a valuation in $S$.
It follows that $v(a_iQ^i)=iv(Q)+v(a_i)\geq iv(Q)>mv(Q)$ for $i=1,\cdots,m-1$. If $v(a_m)=0$ for some $v$, we have $v(\sum_{i=0}^ma_iQ^i)=mv(Q)<0$ which is a contradiction. It follows that $v(a_m)>0$ for all $v\in S.$

\smallskip

(iii)$\Rightarrow$(iv).
By assumption there exists a polynomial $P\in R_S$ and a valuation $v_0\in S$ for which $v_0(P)>0$. It follows that $\Supp ( \Delta \log^+|P| )\not\subseteq S$.
Since we have $S \subset  B (\Supp \,\Delta \log^+|P|)$, Proposition \ref{proapoxishplus} implies the existence of $\phi\in \SHP$ such that $\phi(v)=0$ for all $v\in B(S)$. And we get $\langle\phi,\phi\rangle>0$ as required.

\smallskip

The proof of the  implication (v)$\Rightarrow$(vi) is the core of our Theorem \ref{thmsuvsectionprorefshr}.
We state it as a separate Proposition \ref{prorefshroer} and prove it below.
\endproof

\begin{pro}\label{prorefshroer}Let $S$ be a finite subset of $V_{\infty}$. Suppose that there exists a function $\phi\in \mathbb{L}^2(V_{\infty})$ such that $\phi(v)=0$
for all $v\in B(S)$,
and
$\langle\phi,\phi\rangle>0$.

Then there exists a finite set $S'$ of divisorial valuations such that $S\subseteq B(S')^{\circ}$ and $\Frac(R_{S'})=k(x,y)$.
\end{pro}
The proof relies  on the following lemma that is a corollary of~\cite[Proposition 3.2]{Schroer2000}.
For the convenience of the reader, we give a simplified proof of it at the end of this section.

%

\begin{lem}\label{lemofschroer}Let $X$ be any smooth projective compactification of $\mathbb{A}^2_k$. Let $C$ be a reduced curve contained in $X \setminus \mathbb{A}^2_k$, and set
$U:= X\setminus C$.

 If there exists a $\mathbb{R}$-divisor $A$ supported on $C$ such that $A^2>0,$ then the fraction field of the ring of regular functions on $U$ is equal to  $k(x,y)$.
\end{lem}

\proof [Proof of Proposition \ref{prorefshroer}]
We may assume $S=S^{\min}.$
Let $T_S$ be the convex hull of $S\cup \{-\deg\}$.
This is a finite tree.
Write $\phi=\phi_1-\phi_2$ where both functions $\phi_i$ lie in $\SH (V_\infty)$ and satisfy $\langle\phi_i,\phi_i\rangle>-\infty$ for $i=1,2$. By Proposition \ref{prohodgeindexchangorder} and Proposition \ref{prortgephipsi}, there exists a finite tree $T$ containing $T_S$ such that $$\langle R_T(\phi_1),R_T(\phi_2)\rangle\leq \langle \phi_1,\phi_2\rangle+\frac12 \langle \phi, \phi \rangle.$$
Using Proposition~\ref{prohodgeindexchangorder} (i), we get
\begin{eqnarray*}
\langle R_T(\phi_1)-R_T(\phi_2),R_T(\phi_1)-R_T(\phi_2)\rangle
&\geq&
\langle \phi_1,\phi_1\rangle +
\langle \phi_2,\phi_2\rangle-2\langle R_T(\phi_2),R_T(\phi_1) \rangle
\\
&\geq&
\langle \phi_1,\phi_1\rangle +
\langle \phi_2,\phi_2\rangle-2\langle \phi_1,\phi_2\rangle-\frac12 \langle \phi, \phi \rangle
\\
&=& \frac12 \langle\phi,\phi\rangle >0~.
\end{eqnarray*}
Replacing $\phi$ by $R_T(\phi_1)-R_T(\phi_2)$, we may thus assume that
$\phi$ is the difference of two functions $\phi_1, \phi_2 \in \SH (V_\infty)$
such that $\Delta\phi_1$ and  $\Delta\phi_2$ are supported on a finite tree $T$
whose set of vertices is the union of $S$ and a finite set of divisorial valuations.

\smallskip

\begin{pro}\label{proaproxim}
Let $T$ be any finite subtree of $V_\infty$ containing $-\deg$, and $T'$ be any dense subset of $T$. Suppose
$\phi \in \mathbb{L}^2(V_\infty)$ is a function such that
$\Delta \phi$ is supported on $T$ and $\phi(v)\in \mathbb{R}$ for any end point $v$ of $T$.

Then for any $\epsilon>0$
there exists a piecewise linear function $\phi'$ such that
\begin{enumerate}
\item
the support of $\Delta \phi'$ is  a finite collection of  valuations
that belong to $T'$;
\item
$\phi =  \phi'$ at any endpoint of $T$;
\item
$| \langle \phi, \phi \rangle - \langle \phi', \phi' \rangle | \le \epsilon$.
\end{enumerate}
\end{pro}
Applying this lemma to $ \epsilon = \frac12 \langle \phi, \phi \rangle$, and to the set $T'$ consisting of all divisorial valuations lying in $T\setminus S$, we obtain a piecewise linear function $\phi'$ such that $ \langle \phi', \phi' \rangle > 0$ and the properties (1) -- (3) above are satisfied.

Let $S'$ be the set of extremal points of the support of $\Delta \phi'$. Observe that thanks to our choice of $T'$ and the fact that $\phi|_S = 0$, we have $S \subset B(S')^\circ$ and
$\phi'|_S = 0$.

Now pick any smooth projective compactification $X$ of $\mathbb{A}^2_k$
such that any valuation in $\supp\Delta\phi'\cup S'$ has codimension $1$ center in $X$.
Denote by $E_1,\cdots, E_s$ the centers of valuations in $S'$, and by $E_{s+1},\cdots, E_l$ the other irreducible components of $X\setminus \mathbb{A}^2$. Introduce now the $\mathbb{R}$-divisor
$$
A':=\sum_{i=1}^l b_{E_i}\phi'(v_{E_i})E_i~.
$$

By \cite[Lemma A.2.]{Favre2011}, $$(\sum_{j=1}^lb_{E_j}g_{v_{E_i}}(v_{E_j})E_j\cdot E_k)=0$$ when $k\neq i$, and $$(\sum_{j=1}^lb_{E_j}g_{v_{E_i}}(v_{E_j})E_j\cdot E_k)=b_{E_i}^{-1}$$ when $k=i.$ It follows that $\check{E}_i=b_{E_i}\sum_{j=1}^lb_{E_j}g_{v_{E_i}}(v_{E_j})E_j$ for all $i=1,\cdots,l.$

Write $\phi'=\sum_{i=1}^lc_ig_{v_{E_i}}.$ Then we have
$$A'=\sum_{i=1}^l b_{E_i}\phi'(v_{E_i})E_i=\sum_{i=1}^l b_{E_i}\left(\sum_{j=1}^lc_jg_{v_{E_j}}(v_{E_i})\right)E_i$$
$$=\sum_{j=1}^lb_{E_j}^{-1}c_j\left(b_{E_j}\sum_{i=1}^l b_{E_i}g_{v_{E_j}}(v_{E_i})E_i\right)=\sum_{j=1}^lb_{E_j}^{-1}c_j\check{E}_j.$$
It follows that
 $$(A')^2=\left((\sum_{i=1}^l b_{E_i}\phi'(v_{E_i})E_i)\cdot (\sum_{i=1}^lb_{E_i}^{-1}c_i\check{E}_i)\right)$$
 $$=\sum_{i=1}^lc_i\phi'(E_i)=\langle \phi',\phi'\rangle>0.$$

Since $\phi'|_{S'}= 0$ and $S'$ is the set of extremal points of the support of $\Delta \phi'$ it follows that $\phi' (v_{E_i}) =0$ for any $v_{E_i} \in B(S')$.
In other words,  the support $C$ of $A'$ contains no component $C_i$ such that
$v_{C_i} \in B(S')$.
Now pick $P \in \Gamma(X\setminus C, O_X)$. Then $v_{E_j}(P)\geq 0$ for all $j=1,\cdots,s$
 hence $v(P)\geq 0$ for all $v\in B(S')$
and we conclude that
$$\Gamma(X\setminus C, O_X) \subset R_{S'}=\cap_j \{P\in k[x,y]|\,\,v_{E_j}(P)\geq 0\}~.$$
One completes the proof using Lemma~\ref{lemofschroer}.
%
%
%
%
%
%
%
%
%
%
%
%
%
%
%
%
 \endproof

\begin{proof}[Proof of Proposition~\ref{proaproxim}]Write $\phi=\phi_1-\phi_2$ where both functions $\phi_i$ lie in $\SH (V_\infty)$ and satisfy $\langle\phi_i,\phi_i\rangle>-\infty$ for $i=1,2$.

\noindent{\bf Step 1}. We first suppose that all end points of $T$ are contained in $T'.$

For any $n\geq 0,$ let $T_n$ be a subset of $T'$ such that
\begin{points}
\item[$\d$] all end points of $T$ are contained in $T_n;$
\item[$\d$] for any end point $w$ of $T$ and any point $v\in [-\deg,w]$, there exists a point $v'\in [-\deg,w]\cap T_n$ such that $|\alpha(v)-\alpha(v')|\leq 1/2^{n+1}.$
\end{points}
For $i=1,2$, let $\phi_i^n$ be the unique piecewise linear function on $T$ such that $\phi_i^n(v)=\phi_i(v)$ for all $v\in T_n$. We extend $\phi_i^n$ to a function on $V_{\infty}$ by $\phi^n_i(v):=\phi^n_i(r_T(v))$ for all $v\in V_{\infty}.$
We see that
\begin{points}
\item $\phi_i^n\in \SH(V_{\infty})$;
\item $\Delta\phi_i^n$ is supported on $T$;
\item $\int_T\Delta \phi_i^n=\int_T\Delta \phi_i$;
\item $0\leq \phi_i^n(v)-\phi_i(v)\leq \int_T\Delta \phi_i/2^{n}$ for all $v\in V_{\infty}$.
\end{points}
Set $\phi^n=\phi^n_1-\phi^n_2$. We have
$$ \langle \phi^n,\phi^n\rangle=\sum_{i=1,2;j=1,2}(-1)^{i+j}\int_T\phi_i^n\Delta \phi_j^n$$
$$=\sum_{i=1,2;j=1,2}(-1)^{i+j}(\int_T\phi_i\Delta \phi_j+\int_T(\phi_i^n-\phi_i)\Delta\phi^n_j+\int_T(\phi_j^n-\phi_j)\Delta\phi_i)$$
$$\geq  \langle \phi,\phi\rangle-2(\int_T(\phi_1^n-\phi_1)\Delta\phi^n_2+\int_T(\phi_2^n-\phi_2)\Delta\phi_1)$$
$$\geq \langle \phi,\phi\rangle-4\int_T\Delta \phi_1\int_T\Delta \phi_2/2^n.$$
Then we have $\langle \phi^n,\phi^n\rangle>0$ for $n$ large enough.
Set $\phi':=\phi^n$, then we conclude our Proposition.

\noindent{\bf Step 2}. We complete the proof by induction on the number $n_T$ of end points of $T$ not contained in $T'$.

 When $n_T=0,$ by Step 1, our Proposition holds.

 \medskip

 When $n_T\geq 1$, there exists an end point $w'$ of $T$ not contained in $T'$.
 There exists an increasing sequence $v_n\in [-\deg,w]$ tending to $w$ satisfying $\phi(v_n)\rightarrow\lim_{v<w,v\rightarrow w}\phi(v)=\phi(w)$.
 Since $w$ is an end point, we may suppose that $T_n:=T\setminus (v_n,w]$ is a finite tree.
 There exists a function $g\in \SHP$ such that $\supp \Delta g\subseteq [-\deg, w]$ and it is strict decreasing on $[-\deg, w]$.
 By replacing $\phi_i$ by $\phi_i+g$ for $i=1,2$, we may suppose that $\phi_i$'s are strict decreasing on $[-\deg, w]$.

 When $\phi(v_n)=\phi(w)$, set $\psi_n:=R_{T_n}\phi$.

 When $\phi(v_n)>\phi(w)$, the function $\phi_1(v)-\phi_2(v_n)$ is decreasing. Observe that $\phi_1(v_m)-\phi_2(v_n)=\phi(v_m)-\phi_2(v_n)+\phi_2(v_m)\rightarrow \phi(w)-\phi_2(v_n)+\phi_2(w)$ when $m\rightarrow \infty.$ Since $\phi_2$ is strict decreasing on $[-\deg, w]$, we have $\phi_2(v_n)>\phi_2(w)$ and then
 there exists $v'\in (v_n,w)$ such that $\phi_1(v')-\phi_2(v_n)=\phi(w)$, set $\psi_n:=R_{T\setminus(v',w]}\phi_1-R_{T_n}\phi_2.$

When $\phi(v_n)<\phi(w)$, by the previous argument for $-\phi$, there exists $v'\in (v_n,w)$ such that $\phi_1(v_n)-\phi_2(v')=\phi(w)$, set $\psi_n:=R_{T_n}\phi_1-R_{T\setminus(v',w]}\phi_2.$



By Proposition \ref{prohodgeindexchangorder} and Proposition \ref{prortgephipsi}, there exists $n\geq 0$ such that $|\langle\psi_n,\psi_n\rangle- \langle \phi, \phi \rangle|\le\varepsilon/2$. Since $T'$ is dense in $T$, there exists $w'\in (v_n,w)\cap T'$ such that $\supp\Delta\psi_n\subseteq T\setminus (v_n,w].$
Apply the induction hypotheses to $\psi_n$, there exists a piecewise linear function $\phi'$ such that
\begin{enumerate}
\item[$\d$]
the support of $\Delta \phi'$ is  a finite collection of  valuations
that belong to $T'$;
\item[$\d$]
$\phi' = \psi_n=\phi$ at any endpoint of $T$;
\item[$\d$]
$| \langle \psi_n, \psi_n \rangle - \langle \phi', \phi' \rangle | \le \epsilon/2$.
\end{enumerate}
It follows that $| \langle \phi, \phi \rangle - \langle \phi', \phi' \rangle | \le \epsilon$ which concludes our Proposition.
\end{proof}

\
\proof[Proof of Lemma \ref{lemofschroer}]

Decompose $A = A^+-A^-$ into its positive and negative parts. Since $(A^+)^2+(A^-)^2-2A^+A^-=A^2>0$, and $A^+A^-\geq 0$, we have $(A^+)^2>0$ or $(A^-)^2>0$. Replacing $A$ by $A^+$ or $A^-$, we may thus suppose that $A$ is effective.

Pertubing slightly the coefficients of $A$, we can also impose that $A$ is a $\mathbb{Q}$-divisor.
Let $A=P+N$ be the Zariski decomposition of $A$, see~\cite[Theorem 2.3.19]{Lazarsfeld}.
Here $P$ is a nef and effective $\mathbb{Q}-$divisor,  $N$ is an effective $\mathbb{Q}-$divisors, and they satisfy
$P\cdot N = 0$ and $N^2 <0$. It follows that $ P^2 \ge P^2+N^2 = A^2$.
Replacing $A$ by a suitable multiple of $P$
we may thus assume that $A$ is an effective nef integral divisor with $A^2 >0$.
Now pick any effective integral divisor $D$ whose support is equal to the union of all components of $X\setminus \A^2_k$ that are not contained in $C$.
For $n$ large enough $nA-D$ is big, hence $H^0(nA-D,X) \neq 0$.
Since
\begin{equation*}
H^0(nA-D, X) = \{P\in k(x,y)|\,\,\text{div}(P)+nA\geq D\}~,
\end{equation*}
we may find $P\in k(x,y)$ such that $\text{div}(P)+nA\geq D$.
Since $A$ is supported on $X\setminus U$ and $D$ is effective, $P$ is a regular function on $U$.  Now pick any polynomial $Q \in k[x,y]$. For $m$ large enough,
$v_E(P^mQ)\ge 0$ for any component $E$ of the support of $D$, which implies
$P^mQ$ to be regular on $U$.
This shows that $Q$ is included in the fraction field of $\Gamma (U, O_X)$
hence the latter is equal to $k(x,y)$.
%
%
\endproof

%

\subsection{Reduction to the case of finite skewness}
Recall that given a finite set $S \subset V_\infty$, we let $S^{\min}_{+}$ be the
subset of $S$ consisting of valuations that are minimal in $S$ and of finite skewness.

Our aim is to prove
\begin{thm}\label{thmstofs}Let $S$ be a finite subset of $V_{\infty}$. Then $S$ is rich if and only if $S^{\min}_{+}$ is rich.
\end{thm}

The proof relies on the following result of independent interest.
\begin{thm}\label{thmgeqm}Let $S$ be a finite set of valuations in $V_{\infty}$. Suppose that there exists a function $\phi\in \SH(V_{\infty})$ such that $\langle\phi,\phi\rangle>0$ and $\phi(v)=0$ for all $v\in B(S)$.

For any integer $l\ge0$,  there exists a real number $M_l\leq 1$ such that for any set $S'$ of valuations such that
\begin{points}
\item[(1)]$S'\setminus B(S)$ has at most $l$ elements and,
\item[(2)] $S'\setminus B(S) \subset \{v\in V_{\infty}|\,\, \alpha(v)\leq M_l\}$,
\end{points}
then there exists a function $\phi'\in \mathbb{L}^2(V_{\infty})$
satisfying $\phi'(v)=0$ for all $v\in B(S')$ and $\langle\phi',\phi'\rangle>0.$
\end{thm}

In the particular case where $S=\emptyset$,  the previous result says the following.
\begin{cor}\label{corgeqm}
For any positive integer $l>0$, there exists a real number $M_l\leq 1$ such that given any valuations $v_1,\cdots,v_l$ satisfying $\alpha(v_i)\leq M_l$, there exists a function $\phi\in \mathbb{L}^2(V_{\infty})$
satisfying $\phi'(v)=0$ for all $v\in B(\{v_1,\cdots,v_l\})$ and $\langle\phi',\phi'\rangle>0.$
\end{cor}

\proof[Proof of Theorem \ref{thmstofs}]
As before, we may suppose that $S=S^{\min}.$

Since $S^{\min}_{+}\subseteq S$, we only have to show the "if" part. Suppose that $S^{\min}_{+}$ is rich, and
set $l=\#(S\setminus S^{\min}_{+})$.
Since $S^{\min}_{+}$ is rich, Theorem \ref{thmsuvsectionprorefshr} implies the existence of a function $\phi\in \SHP$ such that $\langle\phi,\phi\rangle>0$ and $\phi(v)=0$ for all $v\in B(S^{\min}_{+})$.
Since $S \setminus B(S^{\min}_{+}) \subset \{ \alpha = -\infty\}$
Theorem \ref{thmgeqm} then implies the existence of $\phi'\in \mathbb{L}^2(V_{\infty})$
satisfying $\phi'(v)=0$ for all $v\in B(S)$ and $\langle\phi',\phi'\rangle>0.$

We conclude that $S$ is rich by applying Theorem \ref{thmsuvsectionprorefshr} once again.
\endproof

\proof[Proof of Theorem~\ref{thmgeqm}]
We first make a couple of reductions. Let $T_S$ be the convex hull of $S$.
Replacing $\phi$ by $R_{T_{S}}(\phi)$, we may suppose that $\Delta\phi$ is supported on $T_S.$ We can also scale $\phi$ so that $\phi(-\deg)=1$ which implies
$0\leq \phi(v)\leq 1$ for all $v\in V_{\infty}$
since $\phi(v)=0$ for all $v\in B(S)$.

Further, we may apply Theorem~\ref{thmsuvsectionprorefshr} (vi) and suppose
$M_0 := \inf_S \alpha > -\infty$.

To simplify notation, set $r:=\langle\phi,\phi\rangle>0.$

\medskip

We prove the theorem  by induction on $l$. In the case $l=0$, there is nothing to prove.  Suppose that the result  holds for $(l-1)\geq 0$ with $M_{l-1}\leq M_0$, and set $M_{l} :=M_{l-1}-2l/r.$

\smallskip

Suppose $S'$ is a set of valuations satisfying the conditions (1) and (2) of the theorem.
When  $\# (S' \setminus B(S) ) \leq l-1$, we are done since
$M_l<M_{l-1}$. So we have $\# (S' \setminus B(S) ) = l$, and we write
$S' \setminus B(S)=\{v_1,\cdots,v_l\}$.  If there exist
a pair of valuations $v_i,v_j$ such that $\alpha(v_i\wedge v_j)\leq M_{l-1}$, then we may conclude by replacing
$S'$ by $(S'\setminus\{v_i,v_j\})\cup \{v_i\wedge v_j\}$ and using the induction hypothesis.

Whence $\alpha(v_i\wedge v_j)>M_{l-1}$ when $i\neq j$.
For each $i$, let $v_i^0$ be the unique valuation in $V_{\infty}$ such that $v_i^0\leq v_i$ and $\alpha(v_i^0)=M_{l-1}$, so that  $v_i^0\neq v_j^0$ when $i \neq j$.
Define $$\Phi_i=  x_i (g_{v_i} - g_{v_i^0}) \in \mathbb{L}^2(V_{\infty}) \text{ with }
x_i:=\phi(v_i)/(M_{l-1}-\alpha(v_i))~.$$
Observe that $\Phi_i (-\deg) = 0$, $\Delta \Phi_i=  x_i (\delta_{v_i} - \delta_{v_i^0})$;
 $1\ge |\Phi_i| \ge 0$, and that
$\Phi_i (v) =-\phi(v_i)$ 
when $v\ge v_i$.
It follows that
$\langle \Phi_i, \Phi_i \rangle = -x_i\phi(v_i)$ and $\langle \Phi_i, \Phi_j \rangle =0$ when $i\neq j$.

Set $$\phi':=\phi+\sum_{i=1}^l\Phi_i~.$$
Then $\phi' \in \mathbb{L}^2(V_{\infty})$, and it is not difficult to check that $\phi'(v)=0$ for all $v\in B(S')$.
Finally we have
\begin{eqnarray*}
\langle\phi',\phi'\rangle
&=&
\langle\phi',\phi\rangle + \sum_{i=1}^l \langle\phi',\Phi_i\rangle = \langle\phi',\phi\rangle-\sum_{i=1}^lx_i\phi'(v_i^0)
\\
&=&
\langle\phi,\phi\rangle-\sum_{i=1}^lx_i\phi(v_i)
\geq
r-\sum_{i=1}^l\phi(v_i)^2/(M_{l-1}-\alpha(v_i))
\\
&\geq&
r-\sum_{i=1}^l1/(M_{l-1}-\alpha(v_i))\geq r/2>0~,
\end{eqnarray*}
which concludes the proof.
\endproof

\subsection{Proof of the Main Theorem}\label{subsectionproofofthmthmtrandeterbycha}

~ \smallskip

By Lemma \ref{lemreducetosubsetofvin}, Lemma \ref{lemredtosminplus} and Theorem \ref{thmstofs}
we may suppose that $S=S^{\min}_+$.

Denote by $T$  the convex hull of  $S\cup \{-\deg\}$. To simplify notation, set $S=\{v_1,\cdots,v_l\}$ and $v_0:=-\deg.$ Since $\alpha(v_0\wedge v_0)=1>0$,
by Lemma \ref{lemhodgeindextheoremvalu}, we have the following
\begin{lem}\label{lemdirichletontree} The matrix
$[\alpha(v_i\wedge v_j)]_{0\le i,j \le l}$ is invertible, and its  determinant has the same sign as $(-1)^l$.
\end{lem}
We may thus find real numbers $a_0, \ldots , a_l$ such that
$$
\left(\begin{matrix}
 1 & 1 & \ldots & 1\\
 1 & \alpha(v_1) & \ldots & \alpha(v_1\wedge v_l)\\
\ldots & \ldots & \ldots & \ldots  \\
  1 & \alpha(v_1\wedge v_l) & \ldots & \alpha(v_l)\\
 \end{matrix} \right)
\left[
\begin{matrix}
a_0 \\ a_1 \\ \vdots \\ a_l \end{matrix}
\right]
=
\left[
\begin{matrix}
1 \\ 0 \\ \vdots \\ 0 \end{matrix}
\right]
~\eqno (*).$$
%
%
%
\begin{lem}\label{lemtolinear}The subset $S$ is rich if and only if $a_0$ is positive.
\end{lem}
Now observe that
\begin{multline*}
\left(\begin{matrix}
 1 & 1 & \ldots & 1\\
 1 & \alpha(v_1) & \ldots & \alpha(v_1\wedge v_l)\\
\ldots & \ldots & \ldots & \ldots  \\
  1 & \alpha(v_1\wedge v_l) & \ldots & \alpha(v_l)\\
 \end{matrix} \right)
\left(\begin{matrix}
 a_0 & 0 & \ldots & 0\\
 a_1 & 1 & \ldots & 0\\
\ldots & \ldots & \ldots & \ldots  \\
  a_l & 0& \ldots & 1\\
 \end{matrix} \right)
=\\
\left(\begin{matrix}
 1 & 1 & \ldots & 1\\
 0 & \alpha(v_1) & \ldots & \alpha(v_1\wedge v_l)\\
\ldots & \ldots & \ldots & \ldots  \\
  0 & \alpha(v_1\wedge v_l) & \ldots & \alpha(v_l)\\
 \end{matrix} \right)
~, \end{multline*}
hence $a_0 >0$ iff $\chi(S) := (-1)^l\det (\alpha(v_i \wedge v_j)_{1 \le i,j \le l} )>0$ as required.

\proof[Proof of Lemma \ref{lemtolinear}]Set $\phi^* := \sum_0^l a_i g_{v_i}\in \mathbb{L}^2(V_{\infty})$. By (*), we have $\phi^*(-\deg)=1$, $\phi^*(v)=0$ for all $v\in B(S)$ and $\langle\phi^*,\phi^*\rangle=a_0.$

Suppose first that $a_0 >0$ $\langle\phi^*,\phi^*\rangle=a_0>0.$ It follows from Theorem~\ref{prodimtwothengeqz} that  $S$ is rich.

\smallskip

Conversely if $S$ is rich, then again by Theorem~\ref{prodimtwothengeqz}  there exists $\phi\in \SHP$ such that $\phi(v)=0$ for all $v\in B(S)$ and $\langle\phi,\phi\rangle>0.$ By replacing $\phi$ by $R_T(\phi)$, we may suppose that $\Delta\phi$ is supported on $T,$ and by scaling,  that $\phi(-\deg)=1.$

Observe that on each connected component of $T\setminus (S\cup \{-\deg\})$, we have $\Delta (\phi-\phi^*)=\Delta (\phi-\phi^*)=\Delta\phi \ge 0$.
The following lemma is basically the maximum principle for subharmonic functions on finite trees.
\begin{lem}\label{lemmaxprinftree}Let $T$ be a finite subtree in $V_{\infty}$ and $S$ be the set of end points of $T$. Suppose that all points in $S$ are with finite skewness. Let $\phi$ subharmonic function on $T\setminus S$
i.e. $\Delta \phi$ is a positive measure on $T\setminus S$. Then if there exists a point $w\in T\setminus S$ satisfying $\phi(w)=\sup\{\phi(v)|\,\,v\in T\setminus S\}$ then $\phi$ is constant in the connected component containing $w$.
\end{lem}

Since
$\phi-\phi^*(v_i)=0$ for all $i=0,\cdots,l$,
Lemma~\ref{lemmaxprinftree} implies that $\phi-\phi ^*\le 0$ on $T$.
Then we conclude that $$a_0=\int \phi^*\Delta \phi^*\geq \int \phi\Delta \phi^*=\int \phi^*\Delta \phi\geq \int \phi\Delta \phi>0.$$
\endproof
\proof[Proof of Lemma \ref{lemmaxprinftree}]We suppose that there exists a point $w\in T\setminus S$ satisfying $\phi(w)=\sup\{\phi(v)|\,\,v\in T\setminus S\}$.

If $w$ is not a branch point, then there exists open segment $I$ in $T$ containing $w$ such that there are no branch points in $I$. Since $\Delta\phi|_I=\frac{d^2\phi}{dt^2}$, we get that $\phi|_I$ is convex. It follows that $\phi$ is constant on $I$.

If $w$ is a branch point, we have $0\leq\Delta\phi\{w\}=\sum_{\vec{w}}D_{\vec{w}}\phi$ where the sum is over all tangent directions $\overrightarrow{w}$ in $T$ at $w$. Then there exists a direction $\vec{v}$ satisfying $D_{\vec{v}}\phi=\max\{D_{\vec{w}}\phi\}$ where the $\max$ is over all tangent directions $\overrightarrow{w}$ in $T$ at $w$. Then we have $D_{\vec{v}}\phi\geq 0$.
There exists a segment $[w,v')$ determining $\vec{v}$ and containing no branch points except $w$. Since $\phi$ is convex on $[w,v')$ and $D_{\vec{v}}\phi\geq 0$, it follows that $\phi$ is constant on $[w,v')$ and then $D_{\vec{w}}\phi=0$ for all tangent directions $\overrightarrow{w}$ in $T$ at $w$. We conclude that there exists an open set $U$ in $T$ containing $w$ such that $\phi$ is constant on $U$.

So the set $\{w|\,\,\sup\{\phi(v)|\,\,v\in T\setminus S\}\}$ is both open and closed. It is thus a union of connected components of $Y\setminus S$ which concludes our lemma.
\endproof

\newpage
\section{Further remarks in the case $\chi(S) = 0$}\label{sectionrechisez}
In this section, we discuss the case when $\chi(S) =0$
for some finite subset $S$ of valuations in $V_{\infty}$, and explore its relations with
the condition $\delta(S)=1$.

As before, $k$ is any algebraically closed field. To simplify the discussion we shall always assume that $S = S^{\min}$, that is
no two different valuations in $S$ are comparable.

\subsection{Characterization of finite sets with $\chi(S)=0$}

\begin{thm}If any valuation in $S$ has finite skewness,
the following  conditions are equivalent:
\begin{enumerate}
\item $\chi(S) =0$;
\item there exists $\phi \in \SHP$ such that $\phi|_S =0$, the support of $\Delta \phi$ is equal to $S$, and $\langle \phi, \phi \rangle =0$.
\end{enumerate}
Moreover when either one of these conditions are satisfied, the function $\phi$ as in $(2)$ is unique up to a scalar factor.
If all valuations in $S$ are divisorial and we normalize $\phi$ such that $\phi(-\deg) = +1$
then the mass of $\Delta \phi$ at any point is a rational number.
\end{thm}
\begin{rem}
When $S = S_+$, $\chi(S)=0$ if and only if
the matrix $M(S)$ has a one-dimensional kernel by Lemma \ref{lemhodgeindextheoremvalu}.
\end{rem}
\begin{defi}
When $\chi(S)=0$ and $S=S^{\min}_+$,  let $\phi_S$ be the unique function in $\SHP$ such that
$\phi_S(-\deg) = +1$, $\phi_S|_S =0$, the support of $\Delta \phi_S$ is equal to $S$, and $\langle \phi_S, \phi_S \rangle =0$ as above.
\end{defi}

\begin{proof}Denote by $T$  the convex hull of  $S\cup \{-\deg\}$. To simplify notation, set $S=\{v_1,\cdots,v_l\}$ and $v_0:=-\deg.$

$(1)\Rightarrow (2).$
By Lemma \ref{lemdirichletontree}, there exists $a_0,\cdots,a_n$ such that
We may thus find real numbers $a_0, \ldots , a_l$ such that
$$
\left(\begin{matrix}
 1 & 1 & \ldots & 1\\
 1 & \alpha(v_1) & \ldots & \alpha(v_1\wedge v_l)\\
\ldots & \ldots & \ldots & \ldots  \\
  1 & \alpha(v_1\wedge v_l) & \ldots & \alpha(v_l)\\
 \end{matrix} \right)
\left[
\begin{matrix}
a_0 \\ a_1 \\ \vdots \\ a_l \end{matrix}
\right]
=
\left[
\begin{matrix}
1 \\ 0 \\ \vdots \\ 0 \end{matrix}
\right]
~.$$
As in the proof of the Main theorem, the signature of $a_0$ is the same as $\chi(S)$. It follows that $a_0=0.$
Consider the function $\phi:=\sum_{i=1}^la_ig_{v_i}.$ Observe that $\phi(-\deg)=1$, $\phi|_S=0$ and $\supp\Delta\phi\subseteq S$.
Lemma~\ref{lemmaxprinftree} implies that $\phi>0$ on $T$. Since $\phi$ is piecewise linear on $T$ and $\phi=0$ on $B(S)$, $a_i=\Delta\phi(v_i)>0$ for $i=1,\cdots,l.$
It follows that $\phi\in \SHP$, $\supp\Delta\phi=S$, $\phi|_S=0$ and $\langle \phi, \phi \rangle =\sum_{i=1}^la_i\phi(v_i)=0$.

$(2)\Rightarrow (1).$ Write $\phi=\sum_{i=1}^la_ig_{v_i}$ where $a_i\in \mathbb{R}^+$, $i=1,\cdots,l.$ Since $\phi|_S=0$, we have
$$
\left(\begin{matrix}
\alpha(v_1) & \ldots & \alpha(v_1\wedge v_l)\\
\ldots & \ldots & \ldots  \\
\alpha(v_1\wedge v_l) & \ldots & \alpha(v_l)\\
 \end{matrix} \right)
\left[
\begin{matrix}
a_1 \\ \vdots \\ a_l \end{matrix}
\right]
=
\left[
\begin{matrix}
0 \\ \vdots \\ 0 \end{matrix}
\right]
~.$$
It follows that $\chi(S)=(-1)^l\det (\alpha(v_i \wedge v_j)_{1 \le i,j \le l} )=0.$

Further, Lemma \ref{lemhodgeindextheoremvalu} implies that the rank of the $l\times l$ matrix $[\alpha(v_i \wedge v_j)]_{1 \le i,j \le l}$ is $l-1$.
It follows that the function $\phi$ is unique up to a scalar factor. When all $v_i$, $i=1,\cdots,l$ are divisorial, then all $\alpha(v_i \wedge v_j)$,
$1 \le i,j \le l$ are rational. If we normalize $\phi$ such that $\phi(-\deg) = +1$
then the mass of $\Delta \phi$ at any point is a rational number.
\end{proof}

\subsection{The relation between $\chi(S)=0$ and $\delta(S) =1$}

Let us begin with the following simple consequence of the Main Theorem.
\begin{pro}\label{prodeseochisez}
If $\delta(S) =1$ then $\chi(S) =0$ and $v$ is divisorial for all $v\in S$.
\end{pro}
\begin{rem}The converse of  Proposition \ref{prodeseochisez} is not true.
Let $L_{\infty}$ be the line at infinity of $\mathbb{P}^2_{\mathbb{C}}$. Let $O$ be a point in $L_{\infty}$ and $(u,v)$ be a local coordinate
at $O$ such that locally $L_{\infty}=\{u=0\}$ and $\{v=0\}$ is a line in $\mathbb{P}^2_{\mathbb{C}}$. Let $C$
be a branch of curve at $O$ defined by $(v-u^2)^5-u^3=0$. We blow up 14 times at the center of ( the strict transform of) $C$
and denote by $E$ the last exceptional curve. One can check that $\alpha(v_E)=0$.
By \cite[Example 1.3, Example 2.5]{Mondala}, we have $\delta(\{v_E\})=0.$
\end{rem}

\proof[Proof of Proposition \ref{prodeseochisez}]
Write $S=\{v_1,\cdots,v_l\}.$
Pick any non constant polynomial $Q\in R_S$, and define $\phi:=\log^+|Q| \in \SHP$.
Since $\delta(S) \neq 2$ it follows from Theorem~\ref{prodimtwothengeqz} (iv) that
$\langle\phi,\phi\rangle \le 0$ hence $\langle\phi,\phi\rangle = 0$, and
$\phi(v)=0$ for all $v\in S$.

Suppose  $v_1\in S$ is not divisorial, then there exists $w_1<v_1$ such that $\phi(w_1)=\phi(v_1)=0.$ By Proposition \ref{proapoxishplus} and Proposition \ref{prorefshroer}, we have $S$ is rich which contradicts to our assumption. It follows that $v$ is divisorial for all $v\in S.$

For every $v_1'>v_1$,
By Proposition \ref{proapoxishplus}, the set $S':=\{v_1',v_2,\cdots,v_l\}$ is rich. It follows that $\chi(S')>0$. Let $v_1'\rightarrow v_1$, we have $\chi(S)\geq 0.$ Since $S$ is not rich, we have $\chi(S)\leq 0$ and then $\chi(S)=0.$
\endproof

Our aim is to state a partial converse to the preceding result.
To do so we need to introduce an important invariant that is referred to as the \emph{thinness} of a valuation in~\cite{Favre2011}. Recall that this is unique function $A:V_{\infty}\rightarrow [-2,\infty]$ that is increasing and lower semicontinuous function on $V_{\infty}$ and such that
$$ A(v_E)=\frac1{b_E} \, \left(1 + \ord_E(dx\wedge dy)\right)$$
for any irreducible  component $E$  of $X\setminus \A^2_k$ in any admissible  compactification.

By the very definition we have $A(-\deg) = -2$ and the thinness of any divisorial valuation is a rational number whereas the thinness of any valuation associated to a branch of an algebraic curve is $+\infty$.

%

We can now state the main result of the section.
\begin{pro}
Suppose $\chi(S) =0$, $v$ is divisorial for all $v\in S$ and $\int A\, \Delta \phi_S \le 0$. Then $\delta(S) =1$.
\end{pro}

\begin{proof}Write $S=\{v_1,\cdots,v_l\}$
and $v_i:=v_{E_i}$ for $E_i\in \mathcal{E}.$ Write $\phi_S=\sum_{i=1}^lr_ig_{v_i}$ where $r_i\in \mathbb{Q}^+.$
Let $X$ be a compactification of $\mathbb{A}^2_k$ such that $E_i$ can be realized as an irreducible component of $X\setminus \mathbb{A}^2_k$.
Let $E_X$ be the set of all irreducible components of $X\setminus \mathbb{A}^2_k$.
Set $\theta:=\sum_{E\in E_X}b_E\phi_S(v_E)E=\sum_{i}^lr_ib_{E_i}^{-1}\check{E}_i.$
Then we have
$$(\theta\cdot K_X)=\sum_{i=1}^lr_ib_{E_i}^{-1}(\check{E}_i\cdot K_X)=\sum_{i=1}^lr_ib_{E_i}^{-1}\ord_{E_i}K_X$$
$$=\sum_{i=1}^sr_ib_{E_i}^{-1}(-1+b_{E_i}A(v_i))=-\sum_{i=1}^sr_ib_{E_i}^{-1}+\int A\, \Delta \phi_S<0.$$
There exists $m\in \mathbb{Z}^+$ such that $D:=m\theta$ is a $\mathbb{Z}$ divisor supposed by $X_{\infty}.$
Then we have that $D$ is effective, $D^2=0$ and $(D\cdot K)\leq -1.$
Recall the Riemann-Roch theorem we have $$h^0_X(D)-h^1_X(D)+h^2_X(D)=\chi(O_X)-\left(D\cdot (D-K)\right)/2=\chi(O_X)-(D\cdot K)/2.$$
Since $X$ is rational, we have $\chi(O_X)=1$. Since $D$ is effective, we have $h^2_X(D)=h^0(K_X-D)\leq h^0(K_X)=0.$
It follows that $$h^0_X(D)\geq 1-(D\cdot K)/2>1.$$ Then there exists an element $P\in k(x,y)\setminus k$ such that $\text{div}(P)+D$ is effective.
 Since $D$ is supposed by $X\setminus \mathbb{A}^2_k$, we have $P\in k[x,y]\setminus k.$
 It follows that $$v_i(P)=(b_{E_i})^{-1}\ord_{E_i}(P)\geq -(b_{E_i})^{-1}\ord_{E_i}(D)=-m\phi_S(v_i)=0$$ for all $i=1,\cdots,l.$
\end{proof}
\begin{rem}The condition $\sum_{i=1}^l r_iA(v_i)\leq 0$ is not necessary.
Set $P:=y^2-x^3 \in \mathbb{C}[x,y]$. Consider the pencil $C_{\la}$ consisting
of the affine curves $C_{\la}:=\{P=\la\}\subseteq \mathbb{C}^2$
for $\la\in \mathbb{C}$. We see that $C_{\la}$ has one branch at infinity for every $\la\in \mathbb{C}$. Let $v_{|C|}$ be the normalized valuation defined by
$Q\mapsto 3^{-1}\ord_{\infty}(Q|_{C_{\la}})$ for $\la$ generic. We see that $\alpha(v_{|C|})=0$, $A(v_{|C|})=1/3>0$ and $P\in R_S.$
\end{rem}

\subsection{The structure of $R_S$ when $\delta(S) =1$}

\begin{pro}\label{protreonestruc}Suppose that $\delta(S)=1$.
Then there exists a polynomial $P\in k[x,y]\setminus k$ such that $R_{S}=k[P]$.
\end{pro}
\proof[Proof of Proposition \ref{protreonestruc}]
Set $S=\{v_1,\cdots,v_l\}$ and suppose that $S=S^{\min}.$

If there exists $Q\in k[x,y]$ such that $Q\in \overline{\Frac(R_{S})}\setminus R_S,$ then we have $\sum_{i=1}^da_iQ^i=0$
where $d\geq 1$, $a_i\in R_S$ and $a_d\neq 0$. Since $S$ is not rich, we have $v(a_i)=0$ for all $v\in S$ and $i=1,\cdots,d.$
Since $Q\neq R_S$, there exists $v\in S$ satisfying $v(Q)<0.$ Then we have $v(a_iQ^i)=iv(Q)<0$ for $i=1,\cdots,d$.
It follows that $v(a_iQ^i)=iv(Q)>dv(Q)=v(a_dQ^d)$ for $i=1,\cdots,m-1$. Then we have $v(\sum_{i=0}^da_iQ^i)=dv(Q)<0$ which is a contradiction.
Then we have $$\overline{\Frac(R_{S})}\bigcap k[x,y]=R_{S}.$$


Pick a polynomial $P\in R_S\setminus k$ with minimal degree. If there are infinitely many $r\in k$ such that $P-r$ is not irreducible,
then by \cite[Th\'eor\`eme fundamental]{Najib2005}, there exists a polynomial $Q\in k[x,y]$ and $R\in k[t]$ of degree at least two
satisfying $P=R\circ Q$. Then we have $Q\in \bar{\Frac(R_{S})}\cap k[x,y]=\Frac(R_{S})$ and $\deg(Q)<\deg(P)$
which contradicts the minimality of $\deg (P).$ It follows that there are infinitely many $r\in k$ such that $P-r$ is irreducible.

 If $R_S\neq k[P]$, there exists $R\in R_S\setminus k[P]$ with minimal degree.
Since $R\in \overline{\Frac(R_S)}=\overline{k(P)}$, we have $$\sum_{i=0}^ma_i(P)R^i=0$$ where $m\geq 1$, $a_i\in k[t]$ and $a_m\neq 0$ in $k[t].$
There exists $r\in k$
such that the polynomial $P-r$ is irreducible and $a_m(r)\neq 0.$ We have $$0=(\sum_{i=0}^ma_i(P)R^i)|_{\{P-r=0\}}=\sum_{i=0}^ma_i(r)(R|_{\{P-r=0\}})^i.$$
It follows that $r_1:=R|_{\{P-r=0\}}$ is a constant in $k$. Since $P-r$ is irreducible, there exists $R_1\in k[x,y]$ such that $R-r_1=(P-r)R_1.$
It follows that $$R_1\in k(R,P)\bigcap k[x,y]\subseteq \Frac(R_{S})\bigcap k[x,y]=R_S$$ and $\deg R_1< \deg R$.
Since the degree of $R$ is minimal in $R_{S}\setminus k[P]$, we have $R_1\in k[P].$ Then we have $R=(P-r)R_1+r_1\in k[P]$ which contradicts to our hypotheses.
It follows that $R_S= k[P]$.
\endproof
\newpage

\section{An application to the algebraization problem of analytic curves}\label{sectionapptoalg}
The aim of this section is to prove Theorem \ref{thmanalytictoalgemanyrational}.
\subsection{$K$-rational points on plane curves}
Let $K$ be a number field, $\mathcal{M}_K^{\infty}$ the set of its archimedean places, $\mathcal{M}_K^0$ the set of its non-archimedean places, and $\mathcal{M}_K=\mathcal{M}^{\infty}_K\cup \mathcal{M}^0_K$. For any $v\in \mathcal{M}_K$, denote by $O_v:=\{x\in K|\,\, |x|_v\leq 1 \}$ the ring of $v$-integers and define $O_K:=\{x\in K|\,\, |x|_v\leq 1 \text{ for all }v\in \mathcal{M}_K^0\}$.

Let $S$ be a finite set of places of $K$ containing all archimedean places.
We define the ring of $S$-integers to be $$O_{K,S}=\{x\in K|\, |x|_v\leq 1 \text{ for all }v\in M_K\setminus S\}.$$

Let $X$ be a compactification of $\mathbb{A}^2_K$. We fix an embedding $\mathbb{A}^2_K\hookrightarrow X$. Fix a projective embedding $X\hookrightarrow \mathbb{P}^N$ defined over $K.$ For each place $v\in \mathcal{M}_K$, there exists a distance function $d_v$ on $X$, defined by
$$d_v([x_0:\cdots:x_N],[y_0:\cdots:y_N])=\frac{\max_{0\leq i,j\leq N}|x_iy_j-x_jy_i|_v}{\max_{0\leq i\leq N}|x_i|_v\max_{0\leq j\leq n}|y_j|_v}$$ for any two points $[x_0:\cdots:x_N],[y_0:\cdots:y_N]\in X(K)\subseteq\mathbb{P}^N(K).$
%
%
Let $C$ be an irreducible curve in $X$ which is not contained in $X_{\infty}:=X\setminus \mathbb{A}^2_K$.
\begin{pro}\label{procurveopensetcoverkpoints}Pick any point $q\in C(K)\bigcap X_{\infty}$. For every place $v\in \mathcal{M}_K$, let $r_v$ be a positive real number and set $U_v:=\{p\in \mathbb{A}^2(K_v)|\,\,d_v(q,p)<r_v\}.$ Suppose more over that $r_v=1$ for all places $v$ outside a finite subset $S$ of $\mathcal{M}_K$.
Then the set $C(K)\setminus \cup_{v\in \mathcal{M}_K}U_v$ is finite.
\end{pro}
\proof
We shall prove that $C(K)\setminus \cup_{v\in \mathcal{M}_K}U_v$ is a set of points with bounded heights for a suitable height.

\smallskip

Let $i:\widetilde{C}\rightarrow C$ be the normalization of $C$ and pick a point $Q\in i^{-1}(q)$.

There exists a positive integer $l$ such that $lQ$ is a very ample divisor of $\widetilde{C}$. Choose an embedding $j: \widetilde{C}\hookrightarrow \mathbb{P}^M$ such that $$O=[1:0:\cdots :0]=H_{\infty}\bigcap \widetilde{C}$$  where $H_{\infty}=\{x_M=0\}$ is the hyperplane at infinity. Let $g:\widetilde{C}\rightarrow \mathbb{P}^1$ be the rational map sending $[x_0:\cdots :x_M]\in \widetilde{C}$ to $[x_0:x_M]\in \mathbb{P}^1$. It is a morphism since $\{x_0=0\}\bigcap H_{\infty}\bigcap \widetilde{C}=\emptyset$. It is also finite and satisfying $$g^{-1}([1:0])=H_{\infty}\bigcap\widetilde{C}=[1:0\cdots:0].$$
By base change, we may assume that $\widetilde{C},i,j, g$ are all defined over $K.$

Set $D=\Spec O_K.$
We consider the irreducible scheme $\widetilde{\mathcal{C}}\subseteq \mathbb{P}^M_{D}$ over $D$ whose generic fiber is $\widetilde{C}$ and the irreducible scheme $\mathcal{X}\subseteq \mathbb{P}^N_{D}$ over $D$ whose generic fiber is $X.$
 Then $i$ extends to a map $\iota:\widetilde{\mathcal{C}}\dashrightarrow \mathcal{X}$ over $D$ that is birational onto its image.

For any $v\in \mathcal{M}^0_K$, let $$\mathfrak{p}_v=\{x\in O_v|\,\,v(x)>0\}$$ be a prime ideal in $O_v.$ There is a finite set $T$ consisting of those places $v\in \mathcal{M}^0_K$ such that $\iota$ is not regular along the special fibre $\overline{C}_{O_v/\mathfrak{p}_v}$ at $\mathfrak{p}_v\in D$ or $\overline{C}_{O_v/\mathfrak{p}_k}\bigcap H_{\infty, O_v/\mathfrak{p}_v}\neq \{[1:0:\cdots:0]\}$.

Pick any place $v\in \mathcal{M}^0_K\setminus (S\bigcup T),$ and define
 \begin{align*}
V_v=&\{[1:x_1:\cdots:x_M]\in \widetilde{C}(K)|\,\, |x_i|_v<1 ,i=1,\cdots, M\}\nonumber\\
         =&\{[1:x_1:\cdots:x_M]\in \widetilde{C}(K)|\,\, |x_M|_v<1 \}.
\end{align*}

Since $r_v=1$, for such a place we set $\Omega_v:=\{[1:x]\in \mathbb{P}^1(K)|\,\,|x|_v<1\}$. We have $V_v=g^{-1}(\Omega_v)\bigcap \widetilde{C}(K),$
so that $i^{-1}(U_v\bigcap C(K))\supseteq V_v$ for all $v\in \mathcal{M}^0_K\setminus (S\bigcup T).$

Now choose a place $v\in S\bigcup T$. Since $g^{-1}([1:0])=Q$, we may supose that $r_v>0$ satisfying $i^{-1}(U_v\bigcap C(K))\supseteq g^{-1}(\{[1:x]\in \mathbb{P}^1(K)|\,\,|x|_v<r_v\})$.

\smallskip

By contradiction, we suppose that there exists a sequence $\{p_n=(x_n,y_n)\}_{n\geq 0}$ of distinct $K$-points in $C(K)\bigcap \mathbb{A}^2(K)$.
Since there are only finitely many singular points in $C$, we may suppose that for all $n\geq 0$, $C$ is regular at $p_n$. Set $q_n:=i^{-1}(p_n)$, and $y_n:=g(q_n).$ Since $g$ is finite, we may suppose that the $y_n$'s are distinct. Write $y_n:=[x_n:1]$ so that
$|x_n|_v<r_v$ for all $v\in \mathcal{M}_K$.

We now observe that
$$[K:\mathbb{Q}]h_{\mathbb{P}^1}(y_n)=\sum_{v\in\mathcal{M}_K}n_v\log(\max\{|x_n|_v,1\})$$
                                      $$\leq\sum_{v\in \mathcal{M}_K\setminus \{v\in \mathcal{M}_K\}}n_v\log(\max\{r_v,1\})
                                     $$$$ =\sum_{v\in S\bigcup T}n_v\log(\max\{r_v,1\})$$
where $h_{\mathbb{P}^1}$ denotes the naive height on $\mathbb{P}^1$.
We get a contradiction by Northcott property (see \cite{Silverman1986}).
\endproof

We also have a version of Proposition \ref{procurveopensetcoverkpoints} for $S$-integral points.

Given any finite set of place containing $\mathcal{M}^{\infty}_K$, we say that $(x,y)\in \mathbb{A}^2(K)\subseteq X$ is $S$-integral if $x,y\in O_{K,S}$.

\begin{pro}\label{prosintsqincurtoabranch}Let $\{p_n=(x_n,y_n)\}_{n\geq 0}$ be an infinite set of $S$-integral points lying in $C\bigcap \mathbb{A}^2$. Then for any point $q\in X_{\infty}\bigcap C(K)$, there exists a place $v\in \mathcal{M}_K$ such that there exists an infinite subsequence $\{p_{n_i}\}_{i\geq 1}$ satisfying $p_{n_i}\rightarrow q$ with respect to $d_v$ as $i\rightarrow \infty$.
\end{pro}
\proof[Proof of Proposition \ref{prosintsqincurtoabranch}]
We define $\widetilde{C}$, $i$,$j$,$g$ and $T$ as in the proof of Proposition \ref{procurveopensetcoverkpoints}.

We may suppose that for all $n\geq 0$, $p_n$ is regular in $C$. The $K$-points $q_n:=i^{-1}(p_n)$ are distinct $K$-points in $\widetilde{C}$.

For any $v\in \mathcal{M}^0_K\setminus (S\bigcup T),$ Set
 \begin{align*}
V_v=&\{[1:x_1:\cdots:x_M]\in \widetilde{C}(K)|\,\, |x_i|_v<1 ,i=1,\cdots, M\}\nonumber\\
         =&\{[1:x_1:\cdots:x_M]\in \widetilde{C}(K)|\,\, |x_M|_v<1 \}.
\end{align*}

 We set $\Omega_v:=\{[1:x]\in \mathbb{P}^1(K)|\,\,|x|_v<1\},$ then $V_v=g^{-1}(\Omega_v)\bigcap \widetilde{C}(K).$ It follows that $q_n\not\in V_v$.
Set $[x_n:1]:=g(q_n).$  Then we have $|x_n|_v<1$ for all $v\in \mathcal{M}_K\setminus \{S\bigcup T\}$

Since $g$ is finite, we may suppose that $g(q_n)$'s are distinct.
By Northcott property, we have $h_{\mathbb{P}^1}(g(q_n))\rightarrow \infty$ as $n\rightarrow\infty$.
Observe that
$$[K:\mathbb{Q}]h_{\mathbb{P}^1}(g(q_n))=\sum_{v\in\mathcal{M}_K}n_v\log(\max\{|x_n|_v,0\})$$
                                      $$=\sum_{v\in \mathcal{M}_K\setminus \{S\bigcup T\}}n_v\log(\max\{|x_n|_v,0\})+\sum_{v\in S\bigcup  T}n_v\log(\max\{|x_n|_v,0\})
                                     $$$$ =\sum_{v\in S\bigcup T}n_v\log(\max\{|x_n|_v,0\})$$
Since $S\bigcup T$ is finite, there exists $v\in S\bigcup T$, such that there exists a subsequence $n_i$ such that $\log(\max\{|x_{n_i}|_v,0\})\rightarrow\infty$ as $i\rightarrow\infty.$ Then $g(q_{n_i})\rightarrow [1:0]$ with respect to $d_v$ as $i\rightarrow\infty.$
Since $g^{-1}([1:0])=\{Q\}$, we have $q_{n_i}\rightarrow Q$ and then $p_{n_i}=i(q_{n_i})\rightarrow q$ respect to $d_v$ as $i\rightarrow\infty.$
\endproof

\subsection{The adelic analytic condition in Theorem \ref{thmanalytictoalgemanyrational}}

Let $K$ be a number field.
Recall that $s$ is an adelic branch at infinity defined over $K$ if it is given by the following data.
%
\begin{points}
\item
$s$ is a formal branch based at a point $q\in L_{\infty}(K)$  given in coordinates $x_q, y_q$ as in the introduction by a formal Puiseux series
$y_q=\sum_{j\ge 1}a_{j}x_q^{j/m}\in O_{K,S}[[x_q^{1/m}]]$  for some positive integer $m$ and be a finite set $S$ of places of $K$ containing all archimedean places.
\item
for each place $v\in S$, the radius of convergence of the Puiseux series determining $s$ is positive, i.e.
$\limsup_{j\rightarrow\infty}|a_j|_v^{-m/j}>0$.
\end{points}

Further, we say $s$ is a adelic branch at infinity if it is a adelic branch defined over some number field.
\begin{rem}The definition of adelic branch at infinity does not depend on the choice of affine coordinate in $\mathbb{A}^2_{\overline{\mathbb{Q}}}$.
\end{rem}


\begin{rem}If $C$ is a branch of an algebraic curve at infinity defined over $\overline{\mathbb{Q}}$, then $C$ is adelic.
\end{rem}

An adelic branch need not to be algebraic. Pick a formal Puiseux series $y_q=\sum_{i=1}^{\infty}a_ix_q^{\frac{i}{m}}\in K[[x_q^{\frac{1}{m}}]]$ which comes from a branch at $q\in L_{\infty}(K)$ of an algebraic curve such that all $a_i$'s are non zero. For example $y_q=\sum_{i=1}^{\infty}x_q^i=\frac{x_q}{1-x_q}.$ To each subset $T$ of $\mathbb{Z}^+$, we attach a formal Puiseux series $y_q=\sum_{i\in T}a_ix_q^{\frac{i}{m}}\in K[[x_q^{\frac{1}{m}}]]$ which defines a formal curve $C_T.$ It is easy to check that all $C_T$'s are adelic-analytic curves and $C_T\neq C_{T'}$ if $T\neq T'$.  So the cardinality of set $\{C_T\}_{T\subseteq \mathbb{Z}^+}$ is
$2^{\aleph_0}$. On the other hand, since $\overline{\mathbb{Q}}$ is countable, the set of all branches of algebraic curves at $O$ is countable. Then there exists an adelic-analytic curve $C_T$ for some $T\subseteq \mathbb{Z}^+$ which is not algebraic.

\bigskip


\bigskip

\subsection{Proof of Theorem \ref{thmanalytictoalgemanyrational}}
Let $S$ be a finite set of places of $K$ containing all archimedean places. We may suppose that $s_1,\cdots,s_l$, $l\geq 1$ are adelic branches defined over $K$.  Denote by $q_i$ the center of $s_i$. Write $U_i$ for $U_{q_i}$, $x_i$ (resp. $y_i$) for $x_{q_i}$ (resp. $y_{q_i}$). By changing coordinates, we may suppose that $x_i = 1/x$, $y_i = y/x + c_i$ for some $c_i\in O_{K,S}$.
Suppose that $s_i$ is defined by $y_i=\sum_{j=1}a_{ij}x^{\frac{j}{m_i}}\in O_{K,S}[[x^{\frac{1}{m_i}}]]$ where $m_i$ is a positive integer.
Observe that $C^v(s_i)$ is contained in the ball $\{p\in \mathbb{P}^2(K_v)|\,\, d_v(p,q_i)<1\}$ for $v\in \mathcal{M}_K\setminus S$.
We may suppose that $B_v=1$ for $v\in M_K\setminus S$.

Since $\alpha(v_{s_i})=-\infty$, by Theorem \ref{thmgeqm} and Theorem \ref{thmsuvsectionprorefshr}, there exists a polynomial $P\in \overline{\mathbb{Q}}[x,y]$ such that $v_i(P)>0$ for all $i=1,\cdots,l.$ Replacing $K$ by a larger number field and $S$ by a larger set, we may suppose that $P\in O_{K,S}[x,y]$.

Observe that $P(x,y)=P(x_i^{-1},(y_i-c_i)x_i^{-1})$ in $U_i$, so that $$P|_{s_i}=P\left(x_i^{-1},(\sum_{j=1}a_{ij}x^{\frac{j}{m_i}}-c_i)x_i^{-1}\right)$$ is a formal Puiseux series. We may write it as  $\sum_{j}^{\infty}b_{i,j}x_i^{\frac{j}{m_i}}\in K((x_i^{\frac{1}{m_i}}))$. It is easy to see that $b_{i,j}\in O_{K,S}$.
Observe that $q_i$ is not a pole of $P|_{C_i}$. It follows that $b_{i,j}=0$ for $j\leq 0$ and then $P|_{C_i}\in K[[x^{\frac{1}{m_i}}]]$.
There exists a real number $M_v\geq 0$ satisfying $|P(p)|_v\leq M_v$ for all $p\in C^v(s_i)$, $i=1,\cdots, l$ and $v\in \mathcal{M}_K.$ Observe that we may chose $M_v=1$ for $v\in M_K\setminus S$.

There exists a number $R_v$ satisfying $|P(x,y)|_v\leq R_v$ for all $(x,y)\in K^2$ satisfying $|x|_v\leq B_v,|y|_v\leq B_v$. We may chose $R_v=1$ for all $v\in M_K\setminus S$. Set $A_v:=\max\{B_v,M_v\}$, we have $A_v= 1$ for $v\in \mathcal{M}_K\setminus S$.

The height of $P(p_n)$ is $$h(P(p_n))=\sum_{v\in \mathcal{M}_K}\log\{1,|P(p_n)|_v\}$$
$$\leq \sum_{v\in \mathcal{M}_K}\log\{1,A_v\}=\sum_{v\in S}\log\{1,A_v\}<\infty.$$
By Northcott property, the set $T:=\{P(p_n)|\,\,n\geq 0\}$ is finite. We denote by $D$ the curve defined by the equation $\prod_{t\in T}(P(x,y)-t)=0.$ Then $D$ contains the set $\{p_n\}_{n\geq 0}$. Let $C$ be the union of all irreducible components of $D$ which contains infinitely many $p_n$. Then for $n$ large enough, we have $p_n\in C.$

\smallskip

We only have to show that all branches of $C$ at infinity are contained in the set $\{s_1,\cdots,s_l\}$.
By contradiction, we suppose that there exists a branch $Z_1$ of $C$ at infinity which is not contained in $\{s_1,\cdots,s_l\}$. Let $Z$ be the irreducible component containing $Z_1.$ Set $R_Z:=\{p_n\}_{n\geq 0}\bigcap Z$. Then $R_Z$ is an infinite set.
Pick a compactification $X$ of $\mathbb{A}^2_K$ such that all centers $q_i'$ of the strict transforms of $s_i$'s are difference from the center $z$ of the strict transform of $Z_1$.  For every $v\in \mathcal{M}_K$ there exists $r_v>0$ such that the ball $D_v:=\{p\in \mathbb{P}^2(K_v)|\,\, d_v(p,z)<r_v\}$ does not intersect $C^v(s_i)$ for all $i=1,\cdots,l$ and does not in intersect the set $\{(x,y)\in \mathbb{A}^2(K_v)|\,\,\max\{|x|_v,|y|_v\}\leq B_v\}.$
 Moreover we may suppose that $r_v=1$ for all $v$ outside a finite set $F$ of $\mathcal{M}_K$. Let $U_v:= D_v\bigcap Z(K_v).$ By Proposition \ref{procurveopensetcoverkpoints}, we have the set $Z(K)\setminus (\cup_{v\in \mathcal{M}_K}U_v)$ is finite. Then there exists a point $p_n\in R_Z$ and a place $v\in \mathcal{M}_K$ such that $p_n=(x_n,y_n)\in U_v.$ Then we have $\max\{|x_n|_v,|y_n|_v\}>B_v$ and $p\not\in C^v(s_i)$ for all $i=1,\cdots,l$, which contradicts to our hypotheses.
\endproof

\begin{rem}\label{remstrongversion}In fact, we can prove a stronger version of Theorem \ref{thmanalytictoalgemanyrational}.
Our proof actually shows that it is only necessary to assume that $p_n$ is a sequence of $\bar{\mathbb{Q}}$ points having bounded degree over $\mathbb{Q}$ (instead of assuming it to belong to the same number field).
%
\end{rem}


We also have an analogue of Theorem \ref{thmanalytictoalgemanyrational} for $S$-integer points.
\begin{thm}\label{thmanalytoalgsinteger}
Let $K$ be a number field and $S$ be a finite subset of places in $\mathcal{M}_K$ containing $\mathcal{M}^{\infty}_K$.

Let $s_1,\cdots,s_l$ where $l\geq 1$ be a finite set of  formal curves in $\mathbb{P}^2_{\overline{\mathbb{Q}}}$ define over $K$ whose centers $q_i$'s are $K$-points in the line $L_{\infty}$ at infinity.
Suppose that for all place $v\in S$, $s_i$ is convergence to a $v$-analytic curve $C^v(s_i)$ in a neighbourhood at $q_i$ w.r.t. $v$ for $i=1,\cdots,l$.

Finally let $p_n = (x^{(n)}, y^{(n)})$, $n\geq 0$ be an infinite collection of $S$-integer points in $\mathbb{A}^2(K)$ such that for each place $v \in M_K$ then either
$\max\{|x^{(n)}|_v,|y^{(n)}|_v\}\leq B_v$ or $p_n \in \cup_{i=1}^lC^v(s_i)$.

Then there exists  an algebraic curve $C$ in $\mathbb{A}^2_K$
such that  any branch of $C$ at infinity is contained in the set $\{s_1,\cdots, s_l\}$ and $p_n$ belongs to $C$ for all $n$ large enough.
\end{thm}

The proof of Theorem \ref{thmanalytoalgsinteger} is very similar to the proof of Theorem \ref{thmanalytictoalgemanyrational}. We leave it to the reader.

\newpage

\newpage

\bibliography{dd}
\end{document}